\documentclass[12pt]{article}

\usepackage{amsfonts}
\usepackage{amsmath}
\usepackage{amssymb}
\usepackage{amsthm}
\usepackage{xcolor}
\usepackage{color}

\usepackage{floatrow}
\usepackage{makeidx}
\usepackage{graphicx}
\usepackage{float}
\usepackage{mathtools}
\usepackage[margin=2cm]{geometry}
\usepackage{subcaption}

\usepackage{tikz-cd}
\usetikzlibrary{graphs,decorations.pathmorphing,decorations.markings}

\usepackage{hyperref}  % comment out before uploading to arXiv. 
\newtheorem{theorem}{Theorem}[section]
\newtheorem{prop}[theorem]{Proposition}
\newtheorem{coro}[theorem]{Corollary}
\newtheorem{lemma}[theorem]{Lemma}
\newtheorem{example}{Example}[section]
\newtheorem{defin}{Definition}[section]
\newtheorem{remark}{Remark}[section]

\newcommand{\pg}{\operatorname{PG}}

\title{
Multiscale analysis via pseudo-reversing and applications \\ to manifold-valued sequences
}

\author{Wael Mattar and Nir Sharon}
\date{}

\begin{document}
\maketitle

\noindent\makebox[\linewidth]{\rule{\textwidth}{1pt}}

\textbf{Abstract.} Modeling data using manifold values is a powerful concept with numerous advantages, particularly in addressing nonlinear phenomena. This approach captures the intrinsic geometric structure of the data, leading to more accurate descriptors and more efficient computational processes. However, even fundamental tasks like compression and data enhancement present meaningful challenges in the manifold setting. This paper introduces a multiscale transform that aims to represent manifold-valued sequences at different scales, enabling novel data processing tools for various applications. Similar to traditional methods, our construction is based on a refinement operator that acts as an upsampling operator and a corresponding downsampling operator. Inspired by Wiener’s lemma, we term the latter as the \emph{reverse} of the former. It turns out that some upsampling operators, for example, least-squares-based refinement, do not have a practical reverse. Therefore, we introduce the notion of \emph{pseudo-reversing} and explore its analytical properties and asymptotic behavior. We derive analytical properties of the induced multiscale transform and conclude the paper with numerical illustrations showcasing different aspects of the pseudo-reversing and two data processing applications involving manifolds.

\noindent\makebox[\linewidth]{\rule{\textwidth}{1pt}}


\noindent\makebox[\linewidth]{\rule{\textwidth}{1pt}}


\section{Introduction}

Manifolds have become ubiquitous in modeling nonlinearities throughout many fields, ranging from science to engineering. In addition, the unceasing increase of sophisticated, modern data sets has brought many challenges in the processing of manifold data, including essential tasks like principle component analysis~\cite{mardia2022principal}, interpolation~\cite{samir2019c1}, optimization~\cite{boumal2023introduction}, and even the adaptation of proper neural networks to manifold values, e.g.,~\cite{chakraborty2020manifoldnet, gao2022curvature}. Here, we focus on multiscale transforms as a critical component in many applications over manifold data, see~\cite{cotronei2019level, mattar2023pyramid, rahman2005multiscale}. In particular, we aim to construct and analyze a new multiscale transform for manifold data.

One foundational statement in the Banach algebra theory and harmonic analysis is Wiener's Lemma, e.g.,~\cite[Chapter 5]{grochenig2010wiener}. The lemma deals with the invertibility and spectrum of operators. In particular, in the Banach space of periodic functions with absolutely convergent Fourier series, the lemma suggests that if a function does not vanish, then there exists a pointwise multiplicative inverse with absolutely convergent Fourier series. These functions, having an inverse in this sense, are essential for multiscale construction, and we term them \emph{reversible}. Their inverse also plays a key role, so the need for reversing rises. Unfortunately, in some cases, the reversibility is poorly possible numerically or even impossible by definition. In such cases, we aspire to suggest an alternative.   

In this paper, we introduce the notion of pseudo-reversing and describe it in detail. As a natural implication of the terminology, and as one may mathematically expect, the pseudo-reverse of a reversible function coincides with its unique inverse. Conversely, applying this method to a non-reversible function produces a family of functions, depending on a continuous regularization parameter, with an absolutely convergent Fourier series. Each function approximates the corresponding inverse according to the selected regularization. Then, we study the algebraic properties of the method and introduce a condition number to determine ``how reversible'' functions are. 

Once pseudo-reversing is established, we show its application for analyzing real-valued sequences in a multiscale fashion. In the context of multiscale transforms, the importance of Wiener's Lemma is evoked when associating a refinement operator with a sequence in $\ell_1(\mathbb{Z})$, that is, the space of absolutely convergent real-valued bi-infinite sequences. Specifically, given a refinement that meets the condition number of reversibility, the lemma guarantees the existence of a corresponding decimation operator. Moreover, a direct result from~\cite{strohmer2002four} implies that the calculation of the decimation involves an infinitely supported sequence which in turn can be truncated while maintaining accuracy~\cite{mattar2023pyramid}. The two operators, refinement and decimation, define a pyramid multiscale transform and its inverse transform. 

In this study, we further generalize pyramid multiscale transforms based on a broader class of refinement operators that do not admit matching decimation operators in an executable form. In particular, we use pseudo-reversing to define the pseudo-reverse of a refinement operator. Epitomai of non-reversible operators appear in the least squares refinements introduced in~\cite{dyn2015univariate}. Nevertheless, even with reversible operators, if their reverse conditioning number is poor, we show that it is preferred to establish their associated pyramid with a pseudo-reversing operator.

As one may expect, since our generalization is based on pseudo-reversing, it comes with a cost. We present the analytical properties of the transform and show that the cost emerges in the synthesis algorithm, that is, the inverse transform, and carries undesired inaccuracies. However, we show that under mild conditions, the error is tolerable.

With the new linear multiscale transforms, we show how to adapt them to Riemannian manifold data. First, we demonstrate how the manifold-valued transform enjoys analog results to the linear case. Specifically, we observe how the magnitude of the detail coefficients in the new multiscale representation decays with the scale. Moreover, we estimate the synthesis error from pseudo-reversing analytically for the specific manifolds with a non-negative sectional curvature.

We conclude the paper with numerical illustrations of pseudo-reversing. First, we show how to use it for constructing a decimation operator for a non-reversible subdivision scheme and the resulting multiscale. Then, we move to manifold-valued data and introduce the applications of contrast enhancement and data compression via our transform. The application is made by systematically manipulating the detail coefficients of manifold-valued sequences. Indeed, the numerical results confirm the theoretical findings. All figures and examples were generated using a Python code package that complements the paper and is available online for reproducibility.

The paper is organized as follows. Section~\ref{sec:pseudo-reversing} lays out the notation and definitions regarding pseudo-reversing and related terms. Section~\ref{sec:linearPyramidTrans} introduces the pyramid transform in its linear settings, where in Section~\ref{sec:multiscalingManifoldData} we present the multiscale transform for manifold values. Finally, in Section~\ref{sec:numerical_examples}, we describe the numerical examples.

\section{Pseudo-reversing and polynomials}\label{sec:pseudo-reversing}

In this section, we briefly revisit Wiener's Lemma and present its classical formulation. In the Banach space of functions with absolutely convergent Fourier series, the lemma proposes a sufficient condition for the existence of a pointwise multiplicative inverse within the space. We term the functions enjoying an inverse in this sense as reversible. Next, we introduce the notion of pseudo-reversing as a method to circumvent the potential non-reversibility of polynomials and describe it in detail. Finally, we present a condition number to measure the reversibility of functions and study how pseudo-reversing improves the reversibility of polynomials.

We follow similar notations presented in~\cite[Chapter 5]{grochenig2010wiener}. Let $\mathbb{T}=\{z\in\mathbb{C}: |z|=1\}$ be the unit circle of the complex plane, denote by $\mathcal{A}(\mathbb{T})$ the Banach space consisting of all periodic functions $f(t)=\sum_{k\in\mathbb{Z}}a_ke^{2\pi ikt}$ with coefficients $\boldsymbol{a}\in\ell_1(\mathbb{Z})$. We endow $\mathcal{A}(\mathbb{T})$ with the norm
\begin{equation*}
    \|f\|_{\mathcal{A}} = \|\boldsymbol{a}\|_1 = \sum_{k\in\mathbb{Z}}|a_k|.
\end{equation*}
The space $\mathcal{A}(\mathbb{T})$ becomes a Banach algebra under pointwise multiplication. In particular, $\|fg\|_\mathcal{A}\leq\|f\|_\mathcal{A}\|g\|_\mathcal{A}$ for any $f,g\in\mathcal{A}(\mathbb{T})$. Given a function $f\in\mathcal{A}(\mathbb{T})$, Wiener's Lemma proposes a sufficient condition for the existence of the inverse $1/f$ in the space $\mathcal{A}(\mathbb{T})$. The classical formulation of the lemma is as follows.

\begin{lemma}~\label{lem:Wiener's_lemma}
    \emph{(Wiener's Lemma)}. If $f\in\mathcal{A}(\mathbb{T})$ and $f(t)\neq 0$ for all $t\in\mathbb{T}$, then also $1/f\in\mathcal{A}(\mathbb{T})$. That is, $1/f(t) = \sum_{k\in\mathbb{Z}}b_ke^{2\pi ikt}$ for some $\boldsymbol{b}\in\ell_1(\mathbb{Z})$.
\end{lemma}

Wiener’s original proof~\cite{wiener1932tauberian} uses a localization property and a partition of unity argument. An abstract proof of Lemma~\ref{lem:Wiener's_lemma} is given by Gel’fand theory, see e.g.~\cite{kaniuth2009course}. A simple and elementary proof can be found in~\cite{newman1975simple}.

\begin{defin}~\label{defin:reversible_function}
    A function $f\in\mathcal{A}(\mathbb{T})$ is called \emph{reversible} if $1/f\in\mathcal{A}(\mathbb{T})$. Moreover, $1/f$ is termed the \emph{reverse} of $f$.
\end{defin}

Lemma~\ref{lem:Wiener's_lemma} guarantees that functions that do not vanish on the unit circle are reversible. One primary class of functions that is advantageous to reverse is polynomials. Indeed, in various applications, many approximating operators are uniquely characterized by polynomials, e.g., refinement operators~\cite{dyn1992subdivision}. We hence focus on the reversibility of polynomials in $\mathcal{A}(\mathbb{T})$.

Let $p(z)=a_nz^n+ a_{n-1}z^{n-1}+\dots + a_1z + a_0$ be a polynomial of degree $n\in\mathbb{N}$ with complex-valued coefficients. Without loss of generality, from now on we assume that the coefficients of $p$ sum to 1, that is $p(1)=\sum_{k=0}^na_k=1$. This requirement coincides with the standard notion of center of mass when the coefficients are treated as weights, and is frequent in many tasks in approximation theory, e.g., in interpolation techniques and partition of unity, see~\cite{cavoretto2021adaptive} for instance. Particularly, it is compatible with condition~\eqref{eqn:partition_of_unity}, as we shall see next in the context of refinement operators.

Denote by $\Lambda$ the set of all zeros of $p$, including multiplicities. That is, if $r$ is a root with multiplicity $m\in\mathbb{N}$ then $r$ appears $m$ many times in $\Lambda$. By the complete factorization theorem, we can write
\begin{equation*}~\label{eqn:factorization_of_p}
p(z) = C(p)\prod_{r\in\Lambda}(z-r)
\end{equation*}
where $C(p)$ is the leading coefficient of $p$. This algebraic expression makes a flexible framework to manipulate the zeros of $p$. The following definition introduces the pseudo-reverse of $p$.

\begin{defin}
    For some $\xi>0$, the pseudo-reverse of a polynomial $p$ is defined as
    \begin{equation}~\label{eqn:pseudo_reverse}
        p_\xi^\dagger(z)=\bigg(C(p_\xi^\dagger)\prod_{r\in\Lambda\setminus\mathbb{T}}(z-r) \prod_{r\in\Lambda\cap\mathbb{T}}(z-(1 + \xi) r)\bigg)^{-1}
    \end{equation}
    where $C(p_\xi^\dagger)$ is a constant depending on $\xi$ determined by $p_\xi^\dagger(1)=1$.
\end{defin}

Note that the pseudo-reverse $p_\xi^\dagger$ of a polynomial $p$ is uniquely determined by the constant $\xi$, and is not a polynomial unless $p(z)\equiv1$. Moreover, one can easily see that if a polynomial $p$ does not vanish on the unit circle $\mathbb{T}$, then its pseudo-reverse coincides with its reverse. That is, $p_\xi^{\dagger}=1/p$ for any $\xi$. The requirement that $p_\xi^\dagger(1)=1$ is proposed to ensure the equality $p_\xi^\dagger(z)p(z)=1$ at least for $z=1$. We denote the inverted term on the right-hand side of~\eqref{eqn:pseudo_reverse}, that is, the term within the parentheses, by $p_\xi^{-\dagger}$. Geometrically speaking, $p_\xi^{-\dagger}$ approximates $p$ by displacing its zeros $\Lambda\cap\mathbb{T}$ along the rays connecting the origin with the zeros, outwards, with a displacement parameter $\xi$. Hence, by Wiener's Lemma~\ref{lem:Wiener's_lemma} the polynomial $p_\xi^{-\dagger}$ is reversible and $p_\xi^\dagger\in\mathcal{A}(\mathbb{T})$.

The perturbations between the polynomial coefficients of $p$ and $p_\xi^{-\dagger}$ can be expressed in a closed form with respect to $\xi$, but the analytical evaluations are not essential to our work. However, it is worth mentioning that the perturbations are monotonic with respect to the power of the argument $z$. In particular, the perturbation between the coefficient of $z^j$ in $p$ and the coefficient of $z^j$ in $p_\xi^{-\dagger}$ increases when $j$ decreases. Figure~\ref{diag:pseudo_reversing} illustrates the computations behind pseudo-reversing~\eqref{eqn:pseudo_reverse}.

\begin{figure}[H]
\centering
 \begin{tikzcd}[sep=large]
    p(z) \arrow[dr, "\substack{\text{approximation}}", swap, sloped]{} \arrow[rr, "\text{pseudo-reversing}"] & & p_\xi^\dagger(z) \\
      & p_\xi^{-\dagger}(z) \arrow[ur, "\substack{\text{reversing}}", swap, sloped]
 \end{tikzcd}
 \caption{The computations behind pseudo-reversing polynomials~\eqref{eqn:pseudo_reverse}. To pseudo-reverse a polynomial $p(z)$, we first approximate it with $p_\xi^{-\dagger}(z)$ by pushing its zeros with modulus 1, away from the unit circle, and then reversing to get $p_\xi^{\dagger}(z)$.}
 \label{diag:pseudo_reversing}
\end{figure}

Let us proceed with an analytical example.

% \begin{example}
%     Consider the polynomial $p(z)=(z^2+1)/2$ which vanishes for $z=\pm i\in\mathbb{T}$. The pseudo reverse $p^\dagger_\xi$ of $p$ calculated via~\eqref{eqn:pseudo_reverse} is then
%     \begin{equation*}
%         p_\xi^\dagger(z)=\frac{2 + 2\xi + \xi^2}{z^2+(1+\xi)^2}, \quad \xi>0.
%     \end{equation*}
% \end{example}

\begin{example}\label{example_2}
    Consider the polynomial $p(z)=(z^2+z+1)/3$ which vanishes for $z=-1/2\pm i\sqrt{3}/2\in\mathbb{T}$. The pseudo reverse $p_\xi^\dagger$ of $p$ calculated via~\eqref{eqn:pseudo_reverse} is then
    \begin{equation*}
        p_\xi^\dagger(z)=\frac{3+3\xi+\xi^2}{z^2 + z(1+\xi)+ (1+\xi)^2}, \quad \xi>0.
    \end{equation*}
    Notice that the function $p_\xi^\dagger(z)$ is not a polynomial. Namely, it cannot be realized as $\sum_{k\in\mathbb{Z}}a_kz^k$ with a finite number of non-zero coefficients. Section~\ref{sec:numerical_examples}, and in particular Figure~\ref{fig:decimation_coefficients}, revisit this example and illustrate different numerical aspects of it.
\end{example}

We now present some properties of pseudo-reversing through a series of useful propositions.

\begin{prop}
    If the polynomial $p$ has real coefficients, then so does the polynomial $p_\xi^{-\dagger}$.
\end{prop}
\begin{proof}
    The assumption implies that if $r\in\Lambda$ is a zero of $p$, then so is the conjugate $\Bar{r}\in\Lambda$. Moreover, the nature of pseudo-reversing~\eqref{eqn:pseudo_reverse} preserves the conjugacy property of the zeros of $p_\xi^{-\dagger}$.
\end{proof}

\begin{prop}~\label{prop:pseudo_approximation}
    For any polynomial $p$, the product $p_\xi^{\dagger}p$ converges in norm to $1$ as $\xi$ approaches $0^+$. Namely,
    \begin{equation*}
        \lim_{\xi\rightarrow 0^+}\|p_\xi^{\dagger}p - 1\|_{\mathcal{A}}=0.
    \end{equation*}
\end{prop}
\begin{proof}
    Direct calculations with an explicit consideration of the constant $C(p^\dagger_\xi)$ of~\eqref{eqn:pseudo_reverse} show that
    \begin{align*}
        p^\dagger_\xi(z)p(z) = \prod_{r\in\Lambda\cap\mathbb{T}}\frac{(z-r)(1-(1+\xi)r)}{(1-r)(z-(1+\xi)r)}.
    \end{align*}
    Because the last product is \emph{finite}, we have that the difference $p^\dagger_\xi(z)p(z)-1$ converges to $0$ for all $z\in\mathbb{C}$ as $\xi$ approaches $0^+$. The desired result then follows by the continuity of $\|\cdot\|_\mathcal{A}$.
\end{proof}

The last proposition guarantees that the distance (induced by the norm) between $p_\xi^\dagger p$ and the identity $1$ approaches $0$ as $\xi$ approaches $0^+$.

\begin{prop}\label{prop:xi_infinity}
    If all zeros of the polynomial $p$ are on the unit circle, then $p_\xi^{\dagger}$ converges uniformly to $1$ as $\xi$ approaches $\infty$ on every compact subset of $\mathbb{C}$.
\end{prop}
\begin{proof}
    The assumption implies that $\Lambda\subset\mathbb{T}$. Therefore, as $\xi$ approaches infinity we have
    \begin{align*}
        \lim_{\xi\rightarrow\infty} p_\xi^\dagger(z) = \lim_{\xi\rightarrow\infty} \frac{\prod_{r\in\Lambda}(1-(1+\xi)r)}{\prod_{r\in\Lambda}(z-(1+\xi)r)} = \lim_{\xi\rightarrow\infty} \prod_{r\in\Lambda}\frac{1-r-\xi r}{z-r-\xi r} = 1
    \end{align*}
    for all $z$ in any compact subset of $\mathbb{C}$.
\end{proof}

The notion of pseudo-reversing induces the necessity of proposing a condition number to quantify the reversibility of functions in $\mathcal{A}(\mathbb{T})$. Conventionally, the condition number of a non-reversible function should take the value $\infty$, whereas the ``best'' reversible function should take the value $1$. Inspired by~\cite{strohmer2002four}, the following definition introduces such a condition number.

\begin{defin}~\label{def:condition_number}
    The reversibility condition number $\kappa \colon \mathcal{A}(\mathbb{T})\rightarrow[1,\infty]$ acting on a function $f\in\mathcal{A}(\mathbb{T})$ is given by
    \begin{equation}~\label{eqn:condition_number}
        \kappa(f)=\frac{\sup_{z\in\mathbb{T}}|f(z)|}{\inf_{z\in\mathbb{T}}|f(z)|},
    \end{equation}
    with the convention $\kappa(f)=\infty$ for functions with $\inf_{z\in\mathbb{T}}|f(z)|=0$.
\end{defin}

The nature of this definition implies that $\kappa(f)$ is well defined for any $f\in\mathcal{A}(\mathbb{T})$ and returns values in $[1,\infty]$. Moreover, the condition number $\kappa(f)$ is invariant under operations that preserve the ratio between the supremum and infimum of $f$ over unit circle $\mathbb{T}$. For example, under scaling and rotating; we have $\kappa(\rho e^{i\theta}f)=\kappa(f)$ for any $\theta\in[0,2\pi)$ and $\rho>0$. Furthermore, $\kappa$ is submultiplicative; for any two functions $f$ and $g$ we have $\kappa(fg)\leq \kappa(f)\kappa(g)$. Meaning that, reversing a product of functions would not be worse than reversing each factor solely.

The rationale behind treating $\kappa(f)$ of~\eqref{eqn:condition_number} as a condition number for reversing a function $f\in\mathcal{A}(\mathbb{T})$ lies in the following fact. Each function $f(t)=\sum_{k\in\mathbb{Z}}a_ke^{2\pi ikt}$ with coefficients $\boldsymbol{a}\in\ell_1(\mathbb{Z})$ defines a bi-infinite Toeplitz matrix of the form $M_{f} = (a_{j-k})_{j,k\in\mathbb{Z}}$. In this case, $M_{f}$ has an inverse if and only if $f$ is
reversible. Moreover, we have $(M_f)^{-1}=M_{f^{-1}}$. Overall, there is a clear analogy between our condition number of a function, as presented in Definition~\ref{def:condition_number}, and that of a matrix. Recall that the condition number of a matrix is defined as the ratio of its largest singular value to its smallest singular value. Furthermore, under certain conditions imposed on the Toeplitz matrix $M_f$, see Theorem 2.1 in~\cite{strohmer2002four}, the entries of its inverse exhibit an exponential decay with respect to their position in the matrix. Adapted to the context of pseudo-reversing, we reformulate the result in the following proposition.

\begin{prop}~\label{prop:decay_of_reverse}
    Let $f(t)=\sum_{k\in\mathbb{Z}}a_ke^{2\pi ikt}\in\mathcal{A}(\mathbb{T})$ be a positive function on the unit circle $\mathbb{T}$. Assume that $f$ is $s$-banded, that is, $a_k=0$ for all $|k|>s$. Fix 
    \begin{equation*}
        \lambda=\bigg(\frac{\sqrt{\kappa(f)}-1}{\sqrt{\kappa(f)}+1}\bigg)^{1/s} \quad \text{and} \quad C=\frac{1}{\inf_{t\in\mathbb{T}}|f(t)|}\max\bigg\{1, \frac{(1+\sqrt{\kappa(f)})^2}{2\kappa(f)}\bigg\}
    \end{equation*}
    where $\kappa(f)$ is the condition number~\eqref{eqn:condition_number} of $f$. Denote by $1/f(t)=\sum_{k\in\mathbb{Z}}b_ke^{2\pi ikt}$. Then,
    \begin{equation*}
        |b_k|\leq C\lambda^{|k|}, \quad k\in\mathbb{Z}.
    \end{equation*}
    Moreover, the coefficients $b_k$ are, in general, not banded.
\end{prop}

The last proposition ensures that smaller values of $\kappa(f)$ lead to a faster decay in the coefficients of the reverse $1/f$. Next, the following corollary shows that increasing the parameter $\xi$ of pseudo-reversing~\eqref{eqn:pseudo_reverse} ensures a better reversibility condition number.

\begin{coro}~\label{coro:kappa_of_pseudo}
    Let $p$ be a polynomial with $n\in\mathbb{N}$ zeros all on the unit circle. Then,
    \begin{equation*}
        \kappa(p_\xi^{-\dagger}) = \frac{\sup_{z\in\mathbb{T}}|\prod_{r\in\Lambda}(z-(1 + \xi) r)|}{\inf_{z\in\mathbb{T}}|\prod_{r\in\Lambda}(z-(1 + \xi) r)|}\leq \prod_{r\in\Lambda}\frac{\sup_{z\in\mathbb{T}}|z-(1 + \xi) r|}{\inf_{z\in\mathbb{T}}|z-(1 + \xi) r|} = (1+2/\xi)^n
    \end{equation*}
    where $\xi$ is the pseudo-reversing parameter. In the last equality, the supremum in the numerator is obtained at $z=-r$, while the infimum in the denominator is obtained at $z=r$.
\end{coro}

The first conclusion from  Corollary~\ref{coro:kappa_of_pseudo} is that the more we push the zeros of $p$ away from the unit circle, the more reversible the polynomial $p_\xi^{-\dagger}$ becomes. Consequently, using Proposition~\ref{prop:decay_of_reverse} we can get an estimation of $\|p^\dagger_\xi\|_\mathcal{A}$ with respect to $\xi$.

A natural implication of Corollary~\ref{coro:kappa_of_pseudo} and Proposition~\ref{prop:pseudo_approximation} is that there exists a trade-off between how well we approximate $p$ with $p_\xi^{-\dagger}$, and how reversible $p_\xi^{-\dagger}$ becomes. The question of finding the optimal parameter $\xi$, which simultaneously minimizes the perturbation $p-p_\xi^{-\dagger}$ and $\kappa(p_\xi^{-\dagger})$, can be answered by numerical tests, as we will see in Section~\ref{sec:numerical_examples}. We conclude the section by remarking that pseudo-reversing can be applied to analytic functions. And in the following sections, we will see a use of pseudo-reversing in the context of multiscaling.

\begin{remark}
    Let $f\in\mathcal{A}(\mathbb{T})$ be an analytic function on the unit circle. We describe how to pseudo-reverse $f$ with some arbitrary error. First, we represent $f$ by its Laurent series, which converges uniformly to $f$ in some compact annulus containing the unit circle. Thanks to the analyticity of $f$, its zeros are guaranteed to be isolated. Then, we truncate the power series to get a polynomial approximating $f$ to any desired degree. Finally, the polynomial can be reversed via~\eqref{eqn:pseudo_reverse}. The result can be considered as the pseudo-reverse of $f$.
\end{remark}

%%%%% Background %%%%%

\section{Linear pyramid transform based on pseudo-reversing} \label{sec:linearPyramidTrans}

In this section, we present the quintessential ideas and notations needed to realize our novel pyramid transform. Then, we use the notion of pseudo-reversing from Section~\ref{sec:pseudo-reversing} to introduce the new pyramid transform. Finally, we study its analytical properties.

\subsection{Background}

Multiscale transforms decompose a real-valued sequence $\boldsymbol{c}^{(J)} = \lbrace c^{(J)}_k\in\mathbb{R}\;\big|\; k\in 2^{-J}\mathbb{Z}\rbrace$ given over the dyadic grid of scale $J\in\mathbb{N}$, to a pyramid of the form $\lbrace \boldsymbol{c}^{(0)}; \boldsymbol{d}^{(1)}, \dots, \boldsymbol{d}^{(J)}\rbrace$ where $\boldsymbol{c}^{(0)}$ is a coarse approximation of $\boldsymbol{c}^{(J)}$ given on the integers $\mathbb{Z}$, and $\boldsymbol{d}^{(\ell)}$, $\ell=1,\dots,J$ are the \emph{detail coefficients} associated with the dyadic grids $2^{-\ell}\mathbb{Z}$, respectively. We focus on transforms that involve refinement operators $\mathcal{S}$ as upsampling operators, and decimation operators $\mathcal{D}$ as their downsampling counterparts. Namely, the multiscale analysis is defined recursively by
\begin{equation}~\label{linear_multiscale_analysis}
\boldsymbol{c}^{(\ell-1)} = \mathcal{D}\boldsymbol{c}^{(\ell)}, \quad \boldsymbol{d}^{(\ell)} = \boldsymbol{c}^{(\ell)} - \mathcal{S}\boldsymbol{c}^{(\ell-1)}, \quad \ell=1,\dots,J,
\end{equation}
while the inverse transform, i.e., the multiscale synthesis is given by
\begin{equation}~\label{linear_multiscale_synthesis}
\boldsymbol{c}^{(\ell)} = \mathcal{S}\boldsymbol{c}^{(\ell - 1)} + \boldsymbol{d}^{(\ell)}, \quad \ell=1,\dots,J.
\end{equation}
Practically, the role of the detail coefficients $\boldsymbol{d}^{(\ell)}$, $\ell=1,\dots,J$ in~\eqref{linear_multiscale_analysis} is to store the data needed to reconstruct $\boldsymbol{c}^{(\ell)}$, that is, approximant of $\boldsymbol{c}^{(J)}$ at scale $\ell$, using the coarser approximant $\boldsymbol{c}^{(\ell-1)}$ of the predecessor scale $\ell-1$. Figure~\ref{diag:pyramid_transform} illustrates the iterative calculations of the multiscale transform~\eqref{linear_multiscale_analysis} and its inverse~\eqref{linear_multiscale_synthesis}.

\begin{figure}[H]
 \centering
 \begin{subfigure}{.5\textwidth}
    \centering
    \begin{tikzcd}
        \boldsymbol{c}^{(J)} \arrow[d, "\mathcal{D}"] \arrow[dr, "-"]
        \\
        \boldsymbol{c}^{(J-1)} \arrow[d, "\mathcal{D}"] \arrow[dr, "-"] & \boldsymbol{d}^{(J)} \\
        \boldsymbol{c}^{(J-2)} \arrow[d, dashed] & \boldsymbol{d}^{(J-1)} \\
        \boldsymbol{c}^{(1)} \arrow[d, "\mathcal{D}"] \arrow[dr,"-"]
        \\
        \boldsymbol{c}^{(0)} & \boldsymbol{d}^{(1)}
    \end{tikzcd}
    \caption{}
    \label{Analysis}
    \end{subfigure}%
    \begin{subfigure}{.5\textwidth}
    \centering
    \begin{tikzcd}
        & \boldsymbol{c}^{(J)} \\
        \boldsymbol{d}^{(J)} \arrow[ur, "+"] & \boldsymbol{c}^{(J-1)} \arrow[u, "\mathcal{S}"] \\
        & \boldsymbol{c}^{(2)} \arrow[u, dashed] \\
        \boldsymbol{d}^{(2)} \arrow[ur, "+"] & \boldsymbol{c}^{(1)} \arrow[u, "\mathcal{S}"] \\
        \boldsymbol{d}^{(1)} \arrow[ur, "+"] & \boldsymbol{c}^{(0)} \arrow[u, "\mathcal{S}"]
    \end{tikzcd}
    \caption{}
    \label{Synthesis}
    \end{subfigure}
 \caption{The pyramid transform. On the left, the analysis~\eqref{linear_multiscale_analysis}. On the right, the synthesis~\eqref{linear_multiscale_synthesis}.}
 \label{diag:pyramid_transform}
\end{figure}

Let $\mathcal{S}_{\boldsymbol{\alpha}}$ be a linear univariate subdivision scheme, consisting of two refinement rules and associated with a finitely supported mask $\boldsymbol{\alpha}$. The subdivision scheme $\mathcal{S}_{\boldsymbol{\alpha}}$ is given explicitly by
\begin{equation}~\label{linear_subdivision_scheme}
    \mathcal{S}_{\boldsymbol{\alpha}}(\boldsymbol{c})_j = \sum_{k\in\mathbb{Z}}\alpha_{j-2k}c_k, \quad j\in\mathbb{Z}.
\end{equation}
Applying the refinement $\mathcal{S}_{\boldsymbol{\alpha}}$ on a sequence $\boldsymbol{c}$ associated with the integers, yields a sequence $\mathcal{S}_{\boldsymbol{\alpha}}(\boldsymbol{c})$ associated with the values over the refined grid $2^{-1}\mathbb{Z}$. Depending on the parity of the index $j$, the subdivision scheme~\eqref{linear_subdivision_scheme} can be split into two refinement rules. Namely,
\begin{equation}~\label{two_subdivision_rules}
\mathcal{S}_{\boldsymbol{\alpha}}(\boldsymbol{c})_{2j} = \sum_{k\in\mathbb{Z}}\alpha_{2k}c_{j-k} \quad \text{and} \quad \mathcal{S}_{\boldsymbol{\alpha}}(\boldsymbol{c})_{2j+1} = \sum_{k\in\mathbb{Z}}\alpha_{2k+1}c_{j-k}, \quad j\in\mathbb{Z}.
\end{equation}
The refinement rule~\eqref{linear_subdivision_scheme} is termed interpolating if $\mathcal{S}_{\boldsymbol{\alpha}}(\boldsymbol{c})_{2j} = c_j$ for all $j\in\mathbb{Z}$. Moreover, a necessary condition for the convergence of a subdivision scheme with the refinement rule, see Proposition 2.1 in~\cite{dyn1992subdivision}, is
\begin{equation}~\label{eqn:partition_of_unity}
\sum_{k\in\mathbb{Z}}\alpha_{2k} = \sum_{k\in\mathbb{Z}}\alpha_{2k+1} = 1.
\end{equation}
With this property, the rules~\eqref{two_subdivision_rules} can be interpreted as moving \emph{center of masses} of the elements of $\boldsymbol{c}$. We assume that any refinement mentioned is of convergent refinement operators.

Given a refinement rule $\mathcal{S}_{\boldsymbol{\alpha}}$, we look for a decimation operator $\mathcal{D}$ such that the detail coefficients $\boldsymbol{d}^{(\ell)}$ generated by the multiscale transform~\eqref{linear_multiscale_analysis}, vanish at all even indices. That is, $d^{(\ell)}_{2j}=0$ for all $j\in\mathbb{Z}$ and $\ell=1,\dots,J$. This property is beneficial for many tasks including data compression. If such a decimation operator $\mathcal{D}$ exists and involves a sequence in $\ell_1(\mathbb{Z})$, then $\mathcal{S}_{\boldsymbol{\alpha}}$ is termed \emph{reversible}, and $\mathcal{D}$ is its \emph{reverse}. This terminology will agree with Definition~\ref{defin:reversible_function} as we will see next. Though, we note here that in~\cite{mattar2023pyramid, dyn2021linear}, such refinement is termed even-reversible.

It turns out that the operator $\mathcal{S}_{\boldsymbol{\alpha}}$ is reversible if its corresponding reverse $\mathcal{D}$ is associated with a real-valued sequence $\boldsymbol{\gamma}\in\ell_1(\mathbb{Z})$ and takes the form
\begin{equation}~\label{linear_decimation}
\mathcal{D}_{\boldsymbol{\gamma}}(\boldsymbol{c})_j = \sum_{k\in\mathbb{Z}}\gamma_{j-k}c_{2k}, \quad j\in\mathbb{Z},
\end{equation}
for any real-valued sequence $\boldsymbol{c}$, while $\boldsymbol{\gamma}$ solves the convolutional equation
\begin{equation}~\label{convolutional_equation}
(\boldsymbol{\alpha}\downarrow 2) * \boldsymbol{\gamma} = \boldsymbol{\delta},
\end{equation}
where $\boldsymbol{\alpha}\downarrow 2$ denotes the even elements of $\boldsymbol{\alpha}$, i.e., $(\boldsymbol{\alpha}\downarrow 2)_j = \alpha_{2j}$ for $j\in\mathbb{Z}$, and $\boldsymbol{\delta}$ is the Kronecker delta sequence ($\delta_0=1$ and $\delta_j=0$ for $j\neq0$).

Contrary to the refinement rule~\eqref{linear_subdivision_scheme}, applying the decimation operator~\eqref{linear_decimation} on a sequence $\boldsymbol{c}$ associated with the dyadic grid $2^{-1}\mathbb{Z}$ produces a sequence $\mathcal{D}_{\boldsymbol{\gamma}}(\boldsymbol{c})$ associated with the integers $\mathbb{Z}$, and hence the term \emph{decimation}. Put simply, the decimation operator convolves the sequence $\boldsymbol{\gamma}$ with the even elements of $\boldsymbol{c}$. If a solution $\boldsymbol{\gamma}$ to~\eqref{convolutional_equation} exists, then we call the coefficients of $\boldsymbol{\gamma}$ the \emph{decimation coefficients}. Moreover, if the refinement is interpolating, that is $\boldsymbol{\alpha}\downarrow 2=\boldsymbol{\delta}$, then $\mathcal{D}_{\boldsymbol{\delta}}$ becomes the simple downsampling operator $\downarrow 2$, returning only the even elements of the input sequence. This fact coincides with the construction of other multiscale analysis methods, see for example~\cite{conti2010full, grohs2009interpolatory, rahman2005multiscale}. The following remark is essential to solve~\eqref{convolutional_equation} and makes the key connection to pseudo-reversing~\eqref{eqn:pseudo_reverse} introduced in Section~\ref{sec:pseudo-reversing}.

\begin{remark}~\label{remark:gamma}
We treat the entries of the sequences appearing in~\eqref{convolutional_equation} as the Fourier coefficients of functions in $\mathcal{A}(\mathbb{T})$, and rely on the convolution theorem to solve the equation. In particular, we transfer both sides with the transform $\boldsymbol{c}\rightarrow\boldsymbol{c}(z)=\sum_{k\in\mathbb{Z}}c_kz^k$ to get
\begin{equation}~\label{Z_transform}
    (\boldsymbol{\alpha}\downarrow2)(z)\boldsymbol{\gamma}(z)=1.
\end{equation}
The function $\boldsymbol{c}(z)$ is termed the symbol of $\boldsymbol{c}$. In other words, given a compactly supported refinement mask $\boldsymbol{\alpha}\downarrow2$ defining the symbol $(\boldsymbol{\alpha}\downarrow 2)(z)\in\mathcal{A}(\mathbb{T})$, we look for its reverse $\boldsymbol{\gamma}(z)$, as defined in Definition~\ref{defin:reversible_function}. The solution $\boldsymbol{\gamma}$ of~\eqref{convolutional_equation} is then the absolutely convergent Fourier coefficients of $\boldsymbol{\gamma}(z)$. If $(\boldsymbol{\alpha}\downarrow2)(z)$ is not reversible, then we turn to pseudo-reversing~\eqref{eqn:pseudo_reverse}, with some parameter $\xi$, to make practical use of this notion.
\end{remark}

Using Proposition~\ref{prop:decay_of_reverse}, the solution $\boldsymbol{\gamma}$ of~\eqref{convolutional_equation} does not have a compact support. This elevates computational challenges. However, a recent study~\cite{mattar2023pyramid} has approximated the decimation operator~\eqref{linear_decimation} with operators involving compactly supported coefficients via proper truncation. The study was concluded with decimation operators that are concretely executable, with negligible errors.

Only when the solution $\boldsymbol{\gamma}$ of~\eqref{Z_transform} is obtained, we are able to employ the refinement operator $\mathcal{S}_{\boldsymbol{\alpha}}$ of~\eqref{linear_subdivision_scheme} together with its reverse $\mathcal{D}_{\boldsymbol{\gamma}}$ of~\eqref{linear_decimation} into the multiscale transform~\eqref{linear_multiscale_analysis} as we will see next. By the nature of this construction, we will indeed have $d^{(\ell)}_{2j}=0$ for all $j\in\mathbb{Z}$ and $\ell=1,\dots,J$. Inspired by the necessity of using sequences $\boldsymbol{\alpha}$ and $\boldsymbol{\gamma}$ that do not particularly satisfy~\eqref{convolutional_equation}, we define the linear operator $\pi^{\boldsymbol{\alpha}}_{\boldsymbol{\gamma}}:\ell_{\infty}(\mathbb{Z})\rightarrow\ell_{\infty}(\mathbb{Z})$, mapping bounded real-valued sequences as follows
\begin{equation}~\label{eqn:even_details}
    \pi^{\boldsymbol{\alpha}}_{\boldsymbol{\gamma}}(\boldsymbol{c}) = [(\mathcal{I}-\mathcal{S}_{\boldsymbol{\alpha}} \mathcal{D}_{\boldsymbol{\gamma}})\boldsymbol{c}]\downarrow 2,
\end{equation}
where $\mathcal{I}$ is the identity operator. The operator $\pi^{\boldsymbol{\alpha}}_{\boldsymbol{\gamma}}$ measures the significance of the detail coefficients on the even indices, with one iteration of decomposition~\eqref{linear_multiscale_analysis} when applied to a sequence $\boldsymbol{c}$. Moreover, if $\boldsymbol{\alpha}$ and $\boldsymbol{\gamma}$ satisfy~\eqref{Z_transform}, then $\pi^{\boldsymbol{\alpha}}_{\boldsymbol{\gamma}}$ becomes the trivial zero operator.

\subsection{Multiscaling with pseudo-reverse operators}

In contrast to the multiscale transforms studied in~\cite{mattar2023pyramid, dyn2021linear} where the reversibility of the refinement is required, here we broaden the family of multiscale transforms~\eqref{linear_multiscale_analysis} to include general non-reversible refinement operators. One interesting family of non-reversible refinement operators is the \emph{least squares} introduced in~\cite{dyn2015univariate}. This branch of schemes was derived by fitting local least squares polynomials. We first exploit the idea of pseudo-reversing~\eqref{eqn:pseudo_reverse} to define the pseudo-reverse of a refinement operator.

\begin{defin}~\label{defin:pseudo-reversing_operators}
    Let $\mathcal{S}_{\boldsymbol{\alpha}}$ be a refinement operator as in~\eqref{linear_subdivision_scheme}. For $\xi>0$, the decimation operator $\mathcal{D}_{\boldsymbol{\gamma}}$ of~\eqref{linear_decimation} is the pseudo-reverse of $\mathcal{S}_{\boldsymbol{\alpha}}$ if
    \begin{equation*}
        \boldsymbol{\gamma}(z)=(\boldsymbol{\alpha}\downarrow 2)^\dagger_\xi(z),
    \end{equation*}
    where $(\boldsymbol{\alpha}\downarrow2)(z)$ and $\boldsymbol{\gamma}(z)$ are defined as in Remark~\ref{remark:gamma}, i.e. if $\boldsymbol{\gamma}$ is the pseudo-reverse of $\boldsymbol{\alpha}\downarrow 2$.
\end{defin}

In this definition, $\boldsymbol{\gamma}$ depends on the pseudo-reversing parameter $\xi$. We denote by $\widetilde{\boldsymbol{\alpha}}$ the mask of the approximating refinement. Namely, the coefficients of $\widetilde{\boldsymbol{\alpha}}\downarrow2$ are the Fourier coefficients of $(\boldsymbol{\alpha}\downarrow 2)^{-\dagger}_\xi(z)$, while the odd values of $\widetilde{\boldsymbol{\alpha}}$ are equal to the odd values of $\boldsymbol{\alpha}$. Similar to pseudo-reversing functions in $\mathcal{A}(\mathbb{T})$, the pseudo-reverse of a reversible refinement coincides with its reverse. Figure~\ref{diag:pseudo_reversing_operators} illustrates the process of pseudo-reversing refinements, and it is an analogue to Figure~\ref{diag:pseudo_reversing} with operators replacing functions.

\begin{figure}[H]
\centering
 \begin{tikzcd}[sep=large]
    \mathcal{S}_{\boldsymbol{\alpha}} \arrow[dr, "\substack{\text{approximation}}", swap, sloped]{} \arrow[rr, "\text{pseudo-reversing}"] & & \mathcal{D}_{\boldsymbol{\gamma}} \\
      & \mathcal{S}_{\widetilde{\boldsymbol{\alpha}}} \arrow[ur, "\substack{\text{reversing}}", swap, sloped]
 \end{tikzcd}
 \caption{Illustration of pseudo-reversing refinement operators.}
 \label{diag:pseudo_reversing_operators}
\end{figure}

Using Definition~\ref{defin:pseudo-reversing_operators} we are now able to generalize multiscale transforms~\eqref{linear_multiscale_analysis} to include a broader class of refinement rules. 

\begin{defin}
    The multiscale transform based on a non-reversible subdivision scheme $\mathcal{S}_{\boldsymbol{\alpha}}$ is the transform~\eqref{linear_multiscale_analysis} where we use its pseudo-inverse $\mathcal{D}_{\boldsymbol{\gamma}}$, for some $\xi>0$, as a decimation operator. The inverse transform~\eqref{linear_multiscale_synthesis} in this case remains unchanged.
\end{defin}

% \begin{defin}~\label{defin:even_singular_multiscale}
% Multiscale transforms Let $\boldsymbol{c}^{(J)}$ be a real-valued sequence of scale $J
% \in\mathbb{N}$, associated on the dyadic grid $2^{-J}\mathbb{Z}$, and let $\mathcal{S}_{\boldsymbol{\alpha}}$ be a refinement rule~\eqref{linear_subdivision_scheme}. The multiscale transform based on $\mathcal{S}_{\boldsymbol{\alpha}}$ is defined by
% \begin{equation}~\label{even_singular_multiscale}
%     \boldsymbol{c}^{(\ell-1)} = \mathcal{D}_{\boldsymbol{\gamma}}\boldsymbol{c}^{(\ell)}, \quad \boldsymbol{d}^{(\ell)} = \boldsymbol{c}^{(\ell)} - \mathcal{S}_{\boldsymbol{\alpha}}\boldsymbol{c}^{(\ell-1)}, \quad \ell=1,\dots,J,
% \end{equation}
% where $\mathcal{D}_{\boldsymbol{\gamma}}$ is the pseudo-reverse of $\mathcal{S}_{\boldsymbol{\alpha}}$ for some $\xi>0$. The inverse transform of~\eqref{even_singular_multiscale} is defined by iterating~\eqref{linear_multiscale_synthesis} for $\ell=1\dots,J$.
% \end{defin}

The cost of using $\mathcal{S}_{\boldsymbol{\alpha}}$ with its pseudo-reverse $\mathcal{D}_{\boldsymbol{\gamma}}$ in multiscaling~\eqref{linear_multiscale_analysis} arises as a violation of the property of having zero detail coefficients on the even indices. That is, $\pi^{\boldsymbol{\alpha}}_{\boldsymbol{\gamma}}$ of~\eqref{eqn:even_details} is not the zero operator. Consequently, requiring small detail coefficients leads to the necessity of studying the operator norm $\|\pi^{\boldsymbol{\alpha}}_{\boldsymbol{\gamma}}\|_{\text{op}}$, where $\boldsymbol{\alpha}$ and $\boldsymbol{\gamma}$ do not satisfy~\eqref{Z_transform}. By the operator norm we mean $\|\pi^{\boldsymbol{\alpha}}_{\boldsymbol{\gamma}}\|_{\text{op}}=\sup_{\boldsymbol{c}\neq \boldsymbol{0}}\|\pi^{\boldsymbol{\alpha}}_{\boldsymbol{\gamma}}(\boldsymbol{c})\|_{\infty}/\|\boldsymbol{c}\|_{\infty}$. The following proposition provides an upper bound on the operator norm $\|\pi^{\boldsymbol{\alpha}}_{\boldsymbol{\gamma}}\|_{\text{op}}$, which can be realized as a global upper bound on the detail coefficients on the even indices.

\begin{prop}~\label{prop:even_details_error}
Let $\boldsymbol{\alpha}$ be a subdivision scheme mask with a pseudo-reverse $\boldsymbol{\gamma}=(\boldsymbol{\alpha}\downarrow 2)^\dagger_\xi$ for some $\xi>0$. Then the operator $\pi^{\boldsymbol{\alpha}}_{\boldsymbol{\gamma}}$ of~\eqref{eqn:even_details} satisfies
\begin{equation}
    \|\pi^{\boldsymbol{\alpha}}_{\boldsymbol{\gamma}}\|_{\text{op}}
    \leq \|\boldsymbol{\alpha}-\widetilde{\boldsymbol{\alpha}}\|_1\|\boldsymbol{\gamma}\|_1,
\end{equation}
where $\widetilde{\boldsymbol{\alpha}}$ is the approximating mask of $\boldsymbol{\alpha}$.
\end{prop}

\begin{proof}
    For any real-valued bounded sequence $\boldsymbol{c}$, we observe that
    \begin{align*}
        \pi^{\boldsymbol{\alpha}}_{\boldsymbol{\gamma}}(\boldsymbol{c}) & = [(\mathcal{I}-\mathcal{S}_{\boldsymbol{\alpha}} \mathcal{D}_{\boldsymbol{\gamma}})\boldsymbol{c}]\downarrow 2
        \\ & = [(\mathcal{I}-\mathcal{S}_{\widetilde{\boldsymbol{\alpha}}} \mathcal{D}_{\boldsymbol{\gamma}}+\mathcal{S}_{\widetilde{\boldsymbol{\alpha}}} \mathcal{D}_{\boldsymbol{\gamma}}-\mathcal{S}_{\boldsymbol{\alpha}} \mathcal{D}_{\boldsymbol{\gamma}})\boldsymbol{c}]\downarrow 2
        \\ & = \pi^{\widetilde{\boldsymbol{\alpha}}}_{\boldsymbol{\gamma}}(\boldsymbol{c}) + [(\mathcal{S}_{\widetilde{\boldsymbol{\alpha}}} \mathcal{D}_{\boldsymbol{\gamma}}-\mathcal{S}_{\boldsymbol{\alpha}} \mathcal{D}_{\boldsymbol{\gamma}})\boldsymbol{c}]\downarrow 2
        \\ & = [(\mathcal{S}_{\widetilde{\boldsymbol{\alpha}}-\boldsymbol{\alpha}} \mathcal{D}_{\boldsymbol{\gamma}})\boldsymbol{c}]\downarrow 2.
    \end{align*}
    The last equality is obtained by the fact that $\boldsymbol{\gamma}$ is the reverse of $\widetilde{\boldsymbol{\alpha}}\downarrow 2$, and by the linearity of the refinement~\eqref{linear_subdivision_scheme}. Now, by taking the $\ell_\infty$ norm we get
    \begin{align*}
        \|\pi^{\boldsymbol{\alpha}}_{\boldsymbol{\gamma}}(\boldsymbol{c})\|_\infty \leq \|\mathcal{S}_{\widetilde{\boldsymbol{\alpha}}-\boldsymbol{\alpha}} (\mathcal{D}_{\boldsymbol{\gamma}}\boldsymbol{c})\|_\infty 
        \leq \sup_{j\in\mathbb{Z}} \sum_{k}|\widetilde{\alpha}_{j-2k}-\alpha_{j-2k}|\sum_{n}|\gamma_{k-n}|\cdot |c_{2n}|\leq \|\widetilde{\boldsymbol{\alpha}}-\boldsymbol{\alpha}\|_1 \|\boldsymbol{\gamma}\|_1 \|\boldsymbol{c}\|_\infty.
    \end{align*}
    Eventually,
    \begin{equation*}
        \|\pi^{\boldsymbol{\alpha}}_{\boldsymbol{\gamma}}\|_{\text{op}} = \sup_{\boldsymbol{c}\neq\boldsymbol{0}}\frac{\|\pi^{\boldsymbol{\alpha}}_{\boldsymbol{\gamma}}(\boldsymbol{c})\|_\infty}{\|\boldsymbol{c}\|_\infty}\leq \|\widetilde{\boldsymbol{\alpha}}-\boldsymbol{\alpha}\|_1 \|\boldsymbol{\gamma}\|_1,
    \end{equation*}
    as required.
\end{proof}

Proposition~\ref{prop:even_details_error} offers the universal upper bound $\|\widetilde{\boldsymbol{\alpha}}-\boldsymbol{\alpha}\|_1 \|\boldsymbol{\gamma}\|_1$, which depends on $\xi$, for the even detail coefficients of~\eqref{linear_multiscale_analysis}. Indeed, there is a trade-off between the quantities $\|\widetilde{\boldsymbol{\alpha}}-\boldsymbol{\alpha}\|_1$ and $\|\boldsymbol{\gamma}\|_1$. In particular, if $\xi$ grows, then so does the perturbation $\|\widetilde{\boldsymbol{\alpha}}-\boldsymbol{\alpha}\|_1$, and according to Corollary~\ref{coro:kappa_of_pseudo} the norm $\|\boldsymbol{\gamma}\|_1$ gets smaller. We end this subsection with the following proposition concerning $\mathcal{S}_{\widetilde{\boldsymbol{\alpha}}}$.

\begin{prop}~\label{prop:even_regular_scheme_properties}
    If $\mathcal{S}_{\boldsymbol{\alpha}}$ is convergent, then so is $\mathcal{S}_{\widetilde{\boldsymbol{\alpha}}}$ for small values of $\xi$.
\end{prop}
\begin{proof}
    First, we note that $\widetilde{\boldsymbol{\alpha}}$ satisfies the necessary condition for convergence~\eqref{eqn:partition_of_unity}. The reason being is that the odd coefficients of the mask $\widetilde{\boldsymbol{\alpha}}$ are equal to the odd coefficients of $\boldsymbol{\alpha}$, while the even coefficients of $\widetilde{\boldsymbol{\alpha}}$ sum to $1$ due to the constant $C$ appearing in~\eqref{eqn:pseudo_reverse}. As for convergence, we refer to Theorem 3.2 in~\cite{dyn1992subdivision}. Since $\mathcal{S}_{\boldsymbol{\alpha}}$ is convergent, then the refinement rule $\mathcal{S}_{\boldsymbol{\beta}}$, where the mask $\boldsymbol{\beta}$ is determined by $\boldsymbol{\beta}(z)=z(1+z)^{-1}\boldsymbol{\alpha}(z)$, is contractive. In particular, there exists $N\in\mathbb{N}$ such that $\|\mathcal{S}_{\boldsymbol{\beta}}^N\|_\infty < 1$, while the $\infty$-norm of $\mathcal{S}_{\boldsymbol{\beta}}$ is calculated via
    \begin{equation}~\label{eqn:subdivision_norm}
        \|\mathcal{S}_{\boldsymbol{\beta}}\|_{\infty} = \max\big\{ \sum_{k\in\mathbb{Z}}|\beta_{2k}|, \; \sum_{k\in\mathbb{Z}}|\beta_{2k+1}|\big\}.
    \end{equation}

    Now let $\widetilde{\boldsymbol{\beta}}(z)=z(1+z)^{-1}\widetilde{\boldsymbol{\alpha}}(z)$ for some $\xi>0$. Because the coefficients of $\widetilde{\boldsymbol{\alpha}}$ are continuous with respect to the parameter $\xi$, see~\eqref{eqn:pseudo_reverse}, then for small values of $\xi$ the contractivity of $\mathcal{S}_{\widetilde{\boldsymbol{\beta}}}$ is guaranteed. This implies that the subdivision scheme $\mathcal{S}_{\widetilde{\boldsymbol{\alpha}}}$ is convergent for small values of $\xi$.
\end{proof}

\subsection{Analytical properties}

In multiscale analysis, and in time-frequency analysis in general, one usually wants to have a stable and perfect reconstruction. This is useful for many numerical tasks, since, we typically manipulate the detail coefficients and then reconstruct using the inverse transform. Perfect reconstruction and stability guarantee the validity of such algorithms.

In the context of our multiscale transform~\eqref{linear_multiscale_analysis} and its inverse, perfect reconstruction means the ability to set half of the detail coefficients of each layer to zero, without losing any information after the synthesis. This property is beneficial for \emph{data compression} since we can avoid storing half of the information. Therefore, half of the detail coefficients of each layer have to exhibit statistical redundancy.

To provide a more precise bound on the detail coefficients on the even indices, recall that, since $\widetilde{\boldsymbol{\alpha}}$ and $\boldsymbol{\gamma}$ satisfy~\eqref{Z_transform}, then using the pair $\mathcal{S}_{\widetilde{\boldsymbol{\alpha}}}$ and $\mathcal{D}_{\boldsymbol{\gamma}}$ in the multiscale transform~\eqref{linear_multiscale_analysis} results in zero detail coefficients on the even indices. The next lemma provides an upper bound to the error on the even detail coefficients when using $\mathcal{S}_{\boldsymbol{\alpha}}$ and $ \mathcal{D}_{\boldsymbol{\gamma}}$ in~\eqref{linear_multiscale_analysis}. To this end, we introduce the operator $\Delta$, which acts on sequences and computes the maximal consecutive difference. In particular, $\Delta \boldsymbol{c}= \sup_{j\in\mathbb{Z}}|c_{j+1} - c_{j}|$ for any real sequence $\boldsymbol{c}$.

\begin{lemma}~\label{lem:details}
    Let $\boldsymbol{c}^{(J)}$ be a real-valued sequence, and let $\mathcal{S}_{\boldsymbol{\alpha}}$ be a refinement operator~\eqref{linear_subdivision_scheme} with a pseudo-reverse $\mathcal{D}_{\boldsymbol{\gamma}}$ for some $\xi>0$, that is, $\boldsymbol{\gamma}=(\widetilde{\boldsymbol{\alpha}}\downarrow2)^{-1}$, see Definition~\ref{defin:pseudo-reversing_operators}. Denote by $\{\boldsymbol{c}^{(0)};\boldsymbol{d}^{(1)},\dots,\boldsymbol{d}^{(J)}\}$ the multiscale representation~\eqref{linear_multiscale_analysis} of $\boldsymbol{c}^{(J)}$ using $\mathcal{S}_{\boldsymbol{\alpha}}$ and $ \mathcal{D}_{\boldsymbol{\gamma}}$. Then
    \begin{equation}~\label{linear_details_bound}
        \|\boldsymbol{d}^{(\ell)}\downarrow 2\|_\infty \leq L(\xi)\Delta\boldsymbol{c}^{(\ell)},
    \end{equation}
    for any $\ell=1,\dots,J$, where
    \begin{equation}~\label{eqn:constant_L}
        L(\xi) = \|\boldsymbol{\alpha} -\widetilde{\boldsymbol{\alpha}}\|_1K_{\boldsymbol{\gamma}} + \|\boldsymbol{\gamma}\|_1 K_{\boldsymbol{\alpha}-\widetilde{\boldsymbol{\alpha}}},
    \end{equation}
    and $K_{\boldsymbol{c}}:=2\sum_{j\in\mathbb{Z}}|c_j|\cdot|j|$.
\end{lemma}

\begin{proof}
    We first note that $K_{\boldsymbol{\gamma}}$ is certainly finite for any $\xi>0$ due to Proposition~\ref{prop:decay_of_reverse}. The finiteness of $K_{\boldsymbol{\alpha}-\widetilde{\boldsymbol{\alpha}}}$ is guaranteed since $\boldsymbol{\alpha}$ is compactly supported. Now, we explicitly calculate a general term of $\boldsymbol{d}^{(\ell)}$ that is associated with an even index. For $\ell=1,\dots,J$ and $j\in\mathbb{Z}$ we have
    \begin{align*}
        d^{(\ell)}_{2j} & = c^{(\ell)}_{2j} - \sum_{k}\alpha_{2k} c_{j-k}^{(\ell-1)}  = \sum_{k} \alpha_{2k} \big(c_{2j}^{(\ell)} - (\mathcal{D}_{\boldsymbol{\gamma}}\boldsymbol{c}^{(\ell)})_{j-k}\big)\\ & = \sum_{k} \alpha_{2k} \big(c_{2j}^{(\ell)} - \sum_{n}\gamma_{j-k-n}c^{(\ell)}_{2n}\big) = \sum_{k} \alpha_{2k} \sum_{n}\gamma_{j-k-n}(c_{2j}^{(\ell)} - c^{(\ell)}_{2n}).
    \end{align*}
    Now, we denote by $\{\boldsymbol{c}^{(0)};\boldsymbol{q}^{(1)},\dots,\boldsymbol{q}^{(J)}\}$ the multiscale representation~\eqref{linear_multiscale_analysis} of $\boldsymbol{c}^{(J)}$ using $\mathcal{S}_{\widetilde{\boldsymbol{\alpha}}}$ and $\mathcal{D}_{\boldsymbol{\gamma}}$, and rely on the fact that the resulting detail coefficients satisfy $q^{(\ell)}_{2j}=0$ for all $\ell=1,\dots,J$ and $j\in\mathbb{Z}$ since $\widetilde{\boldsymbol{\alpha}}$ and $\boldsymbol{\gamma}$ solve~\eqref{convolutional_equation}. On the other hand, direct calculations of $q^{(\ell)}_{2j}$ similar to the above evaluations yield
    \begin{align*}
        q^{(\ell)}_{2j} & = \sum_{k} \widetilde{\alpha}_{2k} \sum_{n}\gamma_{j-k-n}(c_{2j}^{(\ell)} - c^{(\ell)}_{2n})=0.
    \end{align*}
    Consequently,
    \begin{align*}
        |d^{(\ell)}_{2j}| = |d^{(\ell)}_{2j}-q^{(\ell)}_{2j}| & \leq \sum_{k} |\alpha_{2k}-\widetilde{\alpha}_{2k}| \sum_{n}|\gamma_{j-k-n}|\big|c_{2j}^{(\ell)} - c^{(\ell)}_{2n}\big|
        \\ & \leq \sum_{k} |\alpha_{2k}-\widetilde{\alpha}_{2k}| \bigg(\sum_{n}|\gamma_{j-k-n}||2j - 2n|\bigg)\Delta\boldsymbol{c}^{(\ell)}
        \\ & \leq \sum_{k} |\alpha_{2k}-\widetilde{\alpha}_{2k}| \bigg(\sum_{n}|\gamma_{j-k-n}|\big(|2j - 2k -2n|+|2k|\big)\bigg)\Delta\boldsymbol{c}^{(\ell)}
        \\ & \leq \sum_{k} |\alpha_{2k}-\widetilde{\alpha}_{2k}| \bigg(K_{\boldsymbol{\gamma}}+\|\boldsymbol{\gamma}\|_1\cdot|2k|\big)\bigg)\Delta\boldsymbol{c}^{(\ell)}
        \\ & \leq \big(\|\boldsymbol{\alpha} -\widetilde{\boldsymbol{\alpha}}\|_1K_{\boldsymbol{\gamma}} + \|\boldsymbol{\gamma}\|_1 K_{\boldsymbol{\alpha}-\widetilde{\boldsymbol{\alpha}}}\big)\Delta\boldsymbol{c}^{(\ell)}.
    \end{align*}
\end{proof}

In contrast to Proposition~\ref{prop:even_details_error}, Lemma~\ref{lem:details} offers an upper bound on the even detail coefficients which depends on the analyzed sequence $\boldsymbol{c}^{(J)}$. Although the constant $L(\xi)$ of~\eqref{eqn:constant_L} involves two terms that exhibit an indeterminate form of type $0\cdot\infty$ as $\xi$ approaches $0^+$, numerical tests show that the constant is of reasonable magnitude even for small values of $\xi$. Moreover, in practice, $\xi$ is usually chosen large enough to ensure good decay in $\boldsymbol{\gamma}$, resulting in justifiable values for $\|\boldsymbol{\gamma}\|_1$ and $K_{\boldsymbol{\gamma}}$ that appear in $L(\xi)$. Moreover, if $\boldsymbol{\alpha}$ is reversible and does not require pseudo-reversing, then $L(\xi)$ vanishes for all $\xi$, resulting in zero even detail coefficients as expected.

Lemma~\ref{lem:details} leads to the following insight. The even detail coefficients generated by analyzing sequences $\boldsymbol{c}^{(J)}$ with a large value of $\Delta\boldsymbol{c}^{(J)}$ are prone to storing large values. For example, when $\boldsymbol{c}^{(J)}$ is contaminated with noise. In this case, the multiscale transform becomes less beneficial because the representation stores more data than the original input. On the other hand, if $\boldsymbol{c}^{(J)}$ is assumed to be sampled from a differentiable function, then the even detail coefficients must exhibit a decay with respect to the scale, as the following theorem shows.

\begin{theorem}~\label{thm:details_difference}
    Let $\boldsymbol{c}^{(J)}$ be a real-valued sequence sampled from a differentiable function $f:\mathbb{R}\rightarrow\mathbb{R}$, with a bounded derivative, over the equispaced grid $2^{-J}\mathbb{Z}$. Let $\mathcal{S}_{\boldsymbol{\alpha}}$ be a refinement operator~\eqref{linear_subdivision_scheme} with a pseudo-reverse $\mathcal{D}_{\boldsymbol{\gamma}}$ for some $\xi>0$. Denote by $\{\boldsymbol{c}^{(0)};\boldsymbol{d}^{(1)},\dots,\boldsymbol{d}^{(J)}\}$ the multiscale representation~\eqref{linear_multiscale_analysis} of $\boldsymbol{c}^{(J)}$ using $\mathcal{S}_{\boldsymbol{\alpha}}$ and $ \mathcal{D}_{\boldsymbol{\gamma}}$. Then
    \begin{equation}~\label{differentiable_details_bound}
        \|\boldsymbol{d}^{(\ell)}\downarrow 2\|_\infty \leq P(\xi)(2\|\boldsymbol{\gamma}\|_1)^{-\ell},
    \end{equation}
    for any $\ell=1,\dots,J$, where 
    \begin{equation}~\label{eqn:constant_P}
        P(\xi) = L(\xi)\|\boldsymbol{\gamma}\|^J_1\|f^\prime\|_\infty,
    \end{equation}
    and $L(\xi)$ is the constant of~\eqref{eqn:constant_L}.
\end{theorem}

\begin{proof}
    Since $f$ is differentiable and bounded, then by the mean value theorem, for all $j\in\mathbb{Z}$ and a fixed $J\in\mathbb{N}$, there exists $x_j$ in the open segment connecting the parametrizations of $c^{(J)}_j$ and $c^{(J)}_{j+1}$, such that 
    $$|c^{(J)}_{j+1}-c^{(J)}_{j}| = 2^{-J}|f^\prime (x_j)|.$$
    Taking the $\sup_{j\in\mathbb{Z}}$ over both sides gives the estimation $\Delta \boldsymbol{c}^{(J)} \leq 2^{-J}\|f^\prime\|_\infty$. Now, note that the decimation operator~\eqref{linear_decimation} can be written as $\mathcal{D}_{\boldsymbol{\gamma}}\boldsymbol{c} = \boldsymbol{\gamma} * (\boldsymbol{c}\downarrow 2)$ for any real-valued sequence $\boldsymbol{c}$. Moreover,
    since the convolution $*$ commutes with $\Delta$ we get
    \begin{align*}
        \Delta\boldsymbol{c}^{(\ell-1)} = \Delta(\boldsymbol{\gamma}*(\boldsymbol{c}^{(\ell)}\downarrow 2)) \leq \|\boldsymbol{\gamma}\|_1\Delta(\boldsymbol{c}^{(\ell)}\downarrow 2)\leq 2 \|\boldsymbol{\gamma}\|_1\Delta\boldsymbol{c}^{(\ell)}.
    \end{align*}
    Iterating the latter inequality $J-\ell$ many times we get
    $$\Delta\boldsymbol{c}^{(\ell)} \leq (2 \|\boldsymbol{\gamma}\|_1)^{J-\ell}\Delta\boldsymbol{c}^{(J)}\leq \|\boldsymbol{\gamma}\|_1^J(2\|\boldsymbol{\gamma}\|_1)^{-\ell}\|f^\prime\|_\infty.$$
    This estimation together with~\eqref{linear_details_bound} yields the required result by taking $P(\xi)=L(\xi)\|\boldsymbol{\gamma}\|^J_1\|f^\prime\|_\infty$.
\end{proof}

Theorem~\ref{thm:details_difference} implies that the detrimental effect of using $\mathcal{S}_{\boldsymbol{\alpha}}$ and its pseudo-reverse $\mathcal{D}_{\boldsymbol{\gamma}}$ in~\eqref{linear_multiscale_analysis} is more pronounced on coarse scales, as the upper bound of~\eqref{differentiable_details_bound} becomes more loose when $\ell$ decreases. This is true, however, only when analyzing sequences that are sampled from differentiable curves. Moreover, note that the bound becomes more loose as $\|f^\prime\|_\infty$ of~\eqref{eqn:constant_P} increases. In other words, multiscaling with a refinement rule and its pseudo-reverse is more advantageous when the input sequence does not have rapid changes.

In the following theorem we analyze the reconstruction error, that is, the difference between the synthesized pyramid after discarding the even detail coefficients.

\begin{theorem}~\label{thm:synthesis_theorem}
    Let $\boldsymbol{c}^{(J)}$ be a real-valued sequence, and let $\mathcal{S}_{\boldsymbol{\alpha}}$ be a refinement operator~\eqref{linear_subdivision_scheme} with a pseudo-reverse $\mathcal{D}_{\boldsymbol{\gamma}}$ for some $\xi>0$. Denote by $\{\boldsymbol{c}^{(0)};\boldsymbol{d}^{(1)},\dots,\boldsymbol{d}^{(J)}\}$ the multiscale representation~\eqref{linear_multiscale_analysis} of $\boldsymbol{c}^{(J)}$ using $\mathcal{S}_{\boldsymbol{\alpha}}$ and $ \mathcal{D}_{\boldsymbol{\gamma}}$. Then, there exists $C>0$ such that
    \begin{equation}~\label{synthesis_bound}
        \|\boldsymbol{c}^{(J)} - \boldsymbol{\zeta}^{(J)}\|_\infty \leq C\sum_{\ell=1}^{J}\|\boldsymbol{d}^{(\ell)}\downarrow 2\|_\infty,
    \end{equation}
    where $\boldsymbol{\zeta}^{(J)}$ is the synthesized sequence after setting the even detail coefficients of the multiscale representation to zero. In other words, the reconstruction error is proportional to the cumulative errors studied in~\eqref{linear_details_bound} and~\eqref{differentiable_details_bound}.
\end{theorem}

\begin{proof}
    We first note that $\boldsymbol{c}^{(J)}$ can be fully recovered via~\eqref{linear_multiscale_synthesis} when considering \emph{all} the detail coefficients of its pyramid representation. In other words, the iterations~\eqref{linear_multiscale_analysis} and~\eqref{linear_multiscale_synthesis} are algebraically compatible. In particular,
    \begin{equation*}
        \boldsymbol{c}^{(J)}=\mathcal{S}_{\boldsymbol{\alpha}}^{J}\boldsymbol{c}^{(0)} + \sum_{\ell=1}^J\mathcal{S}_{\boldsymbol{\alpha}}^{J-\ell}\boldsymbol{d}^{(\ell)},
    \end{equation*}
    where $\mathcal{S}_{\boldsymbol{\alpha}}^0=\mathcal{I}$ is the identity operator.

    Now, we consider the pyramid $\{\boldsymbol{c}^{(0)};\boldsymbol{q}^{(1)},\dots,\boldsymbol{q}^{(J)}\}$ where
    \begin{equation*}
        q^{(\ell)}_{j}=\begin{cases}
            d^{(\ell)}_j, & j=2n+1, \\
            0, & j=2n,
            \end{cases} \quad \quad\ell=1,\dots, J, \quad j\in\mathbb{Z},
    \end{equation*}
    and study its synthesis via~\eqref{linear_multiscale_synthesis} using $\mathcal{S}_{\boldsymbol{\alpha}}$. Our objective now is to quantify the distance (induced by the $\ell_\infty$ norm) between the reconstruction and the original input $\boldsymbol{c}^{(J)}$.

    To this end, we denote by 
    \begin{equation*}
        \boldsymbol{\zeta}^{(J)}=\mathcal{S}_{\boldsymbol{\alpha}}^{J}\boldsymbol{c}^{(0)} + \sum_{\ell=1}^J\mathcal{S}_{\boldsymbol{\alpha}}^{J-\ell}\boldsymbol{q}^{(\ell)},
    \end{equation*}
    and observe that
    \begin{equation*}
        \|\boldsymbol{c}^{(J)}-\boldsymbol{\zeta}^{(J)}\|_\infty \leq \sum_{\ell=1}^J\|\mathcal{S}_{\boldsymbol{\alpha}}^{J-\ell}(\boldsymbol{d}^{(\ell)}-\boldsymbol{q}^{(\ell)})\|_\infty \leq \sum_{\ell=1}^J\|\mathcal{S}_{\boldsymbol{\alpha}}^{J-\ell}\|_{\infty}\|\boldsymbol{d}^{(\ell)}-\boldsymbol{q}^{(\ell)}\|_\infty,
    \end{equation*}
    where $\|\mathcal{S}_{\boldsymbol{\alpha}}^{J-\ell}\|_{\infty}$ is the operator norm of the refinement $\mathcal{S}_{\boldsymbol{\alpha}}^{J-\ell}$. See~\eqref{eqn:subdivision_norm} for an explicit formula of the norm.

    The uniform boundedness principle guarantees the existence of a constant $C>0$ such that $\sup_{j\in\mathbb{N}} \|\mathcal{S}_{\boldsymbol{\alpha}}^j\|_{\infty}\leq C$. Now, notice that $\boldsymbol{d}^{(\ell)}$ and $\boldsymbol{q}^{(\ell)}$ may differ only on even indices, and hence $\|\boldsymbol{d}^{(\ell)}-\boldsymbol{q}^{(\ell)}\|_\infty=\|\boldsymbol{d}^{(\ell)}\downarrow 2\|_\infty$. We then immediately get the desired result~\eqref{synthesis_bound}.
\end{proof}

In case where $\mathcal{S}_{\boldsymbol{\alpha}}$ is reversible, the detail coefficients of~\eqref{linear_multiscale_analysis} must have zero values, and hence the reconstruction error of~\eqref{synthesis_bound} vanishes. The takeaway message of Theorem~\ref{thm:synthesis_theorem} is that the reconstruction error is proportional to the violation of the property of having zero detail coefficients. For example, considering Theorem~\ref{thm:details_difference}, the error is controllable when the analyzed sequence is sampled from a differentiable function, in conjunction with a pseudo-reverse decimation operator with a reasonably small parameter $\xi$.

Another way to reduce the reconstruction error~\eqref{synthesis_bound} is as follows. Since the upper bound grows with respect to the scale $J$, it is possible to reduce its value by considering only a few iterations in multiscaling, rather than $J$ times. Specifically, for a fixed number of iterations $1\leq m < J$, we can decompose $\boldsymbol{c}^{(J)}$ via~\eqref{linear_multiscale_analysis} $m$ many times into the pyramid $\{ \boldsymbol{c}^{(J-m)}; \boldsymbol{d}^{(J-m+1)},\dots,\boldsymbol{d}^{(J)}\}$ and get a lower synthesis error, even for sequences with large consecutive changes $\Delta\boldsymbol{c}^{(J)}$. However, this reduction in the number of iterations may come at the expense of the quality of several applications, such as data compression, where the coarse approximation plays a crucial role. The above arguments are encapsulated in the following remark.

\begin{remark}
    Under appropriate regularization of the analyzed sequence $\boldsymbol{c}^{(J)}$, and assuming a sufficiently small pseudo-reversibility parameter $\xi$, the utilization of the multiscale transform~\eqref{linear_multiscale_analysis} and its inverse~\eqref{linear_multiscale_synthesis}, constructed from a non-reversible subdivision scheme, can be theoretically justified for a variety of applications in the linear setting. In particular, when the input sequences are derived from differentiable functions, the detail coefficients exhibit decay, resulting in a controlled reconstruction error~\eqref{synthesis_bound}. Moreover, in practical scenarios, only a limited number of multiscale decompositions is typically required, which further mitigates the reconstruction error. Taking into account the approximation error introduced by truncating the decimation coefficients~\cite{mattar2023pyramid}, numerical experiments confirm that the reconstruction error remains well controlled in practice, as we will see in Section~\ref{sec:numerical_linear_multiscaling}.
\end{remark}

Figure~\ref{fig:two_pyramids} epitomizes the theoretical results of this section; it illustrates the analysis and the synthesis with the pairs $(\mathcal{S}_{\boldsymbol{\alpha}}, \mathcal{D}_{\boldsymbol{\gamma}})$ and $(\mathcal{S}_{\widetilde{\boldsymbol{\alpha}}}, \mathcal{D}_{\boldsymbol{\gamma}})$. We conclude the section with two additional remarks.

\begin{remark}~\label{remark:Bspline}
    The method of pseudo-reversing~\eqref{eqn:pseudo_reverse} can be a remedy for reversible refinements $\mathcal{S}_{\boldsymbol{\alpha}}$ with a bad reversibility condition number. That is, refinements with high values of $\kappa(\boldsymbol{\alpha})$ in~\eqref{eqn:condition_number}. Although their reverses involve decimation coefficients $\boldsymbol{\gamma}\in\ell_1(\mathbb{Z})$, their decay rate may be poor due to Corollary~\ref{coro:kappa_of_pseudo} and thus require a large truncation support~\cite{mattar2023pyramid} for implementation. Indeed, a better decay rate can be enforced by pushing the zeros of $(\boldsymbol{\alpha}\downarrow2)(z)$ that are outside of the unit disk, further from $\mathbb{T}$, in a similar manner to pseudo-reversing. Refinements from the family of B-spline~\cite{de1978practical} subdivision schemes are key examples to such operators, as shown in Section~\ref{sec:numerical_examples}.
\end{remark}

\begin{remark}
    Despite having a perfect synthesis algorithm, the reason we avoid using the pair $(\mathcal{S}_{\widetilde{\boldsymbol{\alpha}}}, \mathcal{D}_{\boldsymbol{\gamma}})$ in multiscale transforms is that the approximated refinement $\mathcal{S}_{\widetilde{\boldsymbol{\alpha}}}$ may not inherit the essential properties of the original $\mathcal{S}_{\boldsymbol{\alpha}}$, e.g., the capability of some refinements to produce smooth limit functions may be lost when approximating $\mathcal{S}_{\boldsymbol{\alpha}}$ with $\mathcal{S}_{\widetilde{\boldsymbol{\alpha}}}$.
\end{remark}

\begin{figure}[H]
\centering
    \begin{tikzcd}[sep=large]
             & & \{ \boldsymbol{c}^{(0)};\boldsymbol{d}^{(1)},\dots,\boldsymbol{d}^{(J)}\} \arrow[drr, "\text{synthesis } \boldsymbol{\approx}", sloped] & &  \\ \boldsymbol{c}^{(J)} \arrow[urr, "\text{analysis with } (\mathcal{S}_{\boldsymbol{\alpha}}\text{,}\mathcal{D}_{\boldsymbol{\gamma}})", sloped] \arrow[drr, "\text{analysis with } (\mathcal{S}_{\widetilde{\boldsymbol{\alpha}}}\text{,}\mathcal{D}_{\boldsymbol{\gamma}})", sloped] & & & & \boldsymbol{c}^{(J)} \\
             & & \{ \boldsymbol{c}^{(0)};\widetilde{\boldsymbol{d}}^{(1)},\dots,\widetilde{\boldsymbol{d}}^{(J)}\} \arrow[urr, "\text{synthesis }\boldsymbol{=}", sloped] & &
    \end{tikzcd}
 \caption{Two pyramid transforms of the sequence $\boldsymbol{c}^{(J)}$. With a refinement operator $\mathcal{S}_{\boldsymbol{\alpha}}$, the upper transform uses the pair $(\mathcal{S}_{\boldsymbol{\alpha}}, \mathcal{D}_{\boldsymbol{\gamma}})$ while the lower transform uses $(\mathcal{S}_{\widetilde{\boldsymbol{\alpha}}}, \mathcal{D}_{\boldsymbol{\gamma}})$. The decimation operator $\mathcal{D}_{\boldsymbol{\gamma}}$ is mutual in both transforms and is the pseudo-reverse of $\mathcal{S}_{\boldsymbol{\alpha}}$.}
 \label{fig:two_pyramids}
\end{figure}

%%%%% Manifold Values %%%%%

\section{Multiscaling manifold values} \label{sec:multiscalingManifoldData}

In this section, we adapt the multiscale transform~\eqref{linear_multiscale_analysis} to manifold-valued sequences. To this purpose, it is necessary to adapt the operators $\mathcal{S}_{\boldsymbol{\alpha}}$ and $\mathcal{D}_{\boldsymbol{\gamma}}$ of~\eqref{linear_subdivision_scheme} and~\eqref{linear_decimation} to manifolds, as well as defining the analogues to the elementary operations `$-$' and `$+$' appearing in the transform and its inverse. Indeed, there are various methods for adapting the multiscale transform; however, we follow the same adaptations and definitions as those in~\cite {mattar2023pyramid}.

Let $\mathcal{M}$ be an open complete Riemannian manifold equipped with Riemannian metric $\langle\cdot,\cdot\rangle$. The Riemannian geodesic distance $\rho:\mathcal{M}^2\rightarrow\mathbb{R}^{+}$ is defined by
\begin{equation}~\label{Riemannian_geodesic}
    \rho(x,y)=\inf_{\Gamma}\int_{a}^b\|\dot{\Gamma}(t)\|dt,
\end{equation}
where $\Gamma:[a,b]\rightarrow\mathcal{M}$ is a piecewise differentiable curve connecting the points $x=\Gamma(a)$ and $y=\Gamma(b)$, and $\|\cdot\|^2=\langle\cdot,\cdot\rangle$.

Based on the Riemannian geodesic~\eqref{Riemannian_geodesic}, the operators $\mathcal{S}_{\boldsymbol{\alpha}}$ and $\mathcal{D}_{\boldsymbol{\gamma}}$ of~\eqref{linear_subdivision_scheme} and~\eqref{linear_decimation} are adapted to $\mathcal{M}$ respectively by the optimization problems
\begin{equation}~\label{manifold_subdivision}
    \mathcal{T}_{\boldsymbol{\alpha}}(\boldsymbol{c})_j = \text{argmin}_{x\in\mathcal{M}}\sum_{k\in\mathbb{Z}}\alpha_{j-2k}\rho^2(x,c_k), \quad j\in\mathbb{Z},
\end{equation}
and
\begin{equation}~\label{manifold_decimation}
    \mathcal{Y}_{\boldsymbol{\gamma}}(\boldsymbol{c})_j = \text{argmin}_{x\in\mathcal{M}}\sum_{k\in\mathbb{Z}}\gamma_{j-k}\rho^2(x,c_{2k}),\quad j\in\mathbb{Z}, 
\end{equation}
for $\mathcal{M}$-valued sequences $\boldsymbol{c}$. When the solution of~\eqref{manifold_subdivision} or~\eqref{manifold_decimation} exists uniquely, we term the solution as \emph{Riemannian center of mass}~\cite{grove1973conjugate}. It is also referred to as the Karcher mean for matrices and the Frèchet mean in more general metric spaces; see~\cite{karcher2014riemannian}.

The global well-definedness of~\eqref{manifold_subdivision} and~\eqref{manifold_decimation} when $\boldsymbol{\alpha}$ and $\boldsymbol{\gamma}$ have non-negative entries is studied in~\cite{kobayashi1996foundations}. Moreover, in the framework where $\mathcal{M}$ has a non-positive sectional curvature, if the mask $\boldsymbol{\alpha}$ satisfies~\eqref{eqn:partition_of_unity}, then a globally unique solution to problem~\eqref{manifold_subdivision} can be found, see, e.g.~\cite{karcher1977riemannian, hardering2015intrinsic, sander2016geodesic}. The same argument applies to problem~\eqref{manifold_decimation} since the elements of $\boldsymbol{\gamma}$ must sum to 1 due to~\eqref{Z_transform}. Recent studies of manifolds with positive sectional curvature show necessary conditions for uniqueness on the spread of points with respect to the injectivity radius of $\mathcal{M}$~\cite{dyer2016barycentric, huning2022convergence}. We assume, without loss of generality, that both~\eqref{manifold_subdivision} and~\eqref{manifold_decimation} are uniquely solved for any $\mathcal{M}$-valued sequences $\boldsymbol{c}$. Otherwise, different regularization conditions can be used to determine unique values.

We say that the nonlinear operator $\mathcal{T}_{\boldsymbol{\alpha}}$ is non-reversible if its linear counterpart $\mathcal{S}_{\boldsymbol{\alpha}}$ is non-reversible. Similarly to pseudo-reversing refinements in the linear setting, as in Definition~\ref{defin:pseudo-reversing_operators}, we define the pseudo-reverse of the manifold-valued refinement $\mathcal{T}_{\boldsymbol{\alpha}}$. Namely, we say that the decimation operator $\mathcal{Y}_{\boldsymbol{\gamma}}$ is the pseudo-reverse of $\mathcal{T}_{\boldsymbol{\alpha}}$ if for some $\xi>0$, we have $\boldsymbol{\gamma}$ is the pseudo-reverse of $\boldsymbol{\alpha}$. In short, Figure~\ref{diag:pseudo_reversing_operators} illustrates the calculations behind pseudo-reversing refinement operators in the manifold setting, where linear operators are replaced with nonlinear ones.

Given a Riemannian manifold $\mathcal{M}$, recall that the exponential mapping $\text{exp}_p$ maps a vector $v$ in the tangent space $T_p\mathcal{M}$ to the end point of a geodesic of length $\|v\|$, which emanates from $p\in\mathcal{M}$ with an initial tangent vector $v$. Inversely, $\text{log}_p$ is the inverse map of $\text{exp}_p$ that takes an $\mathcal{M}$-valued element $q$ and returns a vector in the tangent space $T_p\mathcal{M}$. Following similar notations used in~\cite{grohs2009interpolatory} and~\cite{mattar2023pyramid}, we denote both maps by
\begin{equation}\label{manifold_operations}
    \text{log}_p(q)=q\ominus p\quad \text{and} \quad \text{exp}_p(v)=p\oplus v.
\end{equation}
We have thus defined the analogues $\ominus$ and $\oplus$ of the `$-$' and `$+$' operations appearing in~\eqref{linear_multiscale_analysis}, respectively. For any point $p\in\mathcal{M}$ we use the following notation $\ominus:\mathcal{M}^2\rightarrow T_p\mathcal{M}$ and $\oplus:\mathcal{M}\times T_p\mathcal{M}\rightarrow\mathcal{M}$. Then, the compatibility condition is
\begin{equation*}
    (p \oplus v) \ominus p = v,
\end{equation*}
for all $v\in T_p\mathcal{M}$ within the injectivity radius of $\mathcal{M}$. With the operators~\eqref{manifold_subdivision},~\eqref{manifold_decimation}, and~\eqref{manifold_operations} in hand, we are able to define the analogue of the multiscale transform~\eqref{linear_multiscale_analysis}.

\begin{defin}~\label{defin:manifold_even_singular_multiscale}
Let $\mathcal{M}$ be a Riemannian manifold, and let $\boldsymbol{c}^{(J)}$ be an $\mathcal{M}$-valued sequence of scale $J\in\mathbb{N}$ parametrized over the dyadic grid $2^{-J}\mathbb{Z}$, and let $\mathcal{T}_{\boldsymbol{\alpha}}$ be a refinement rule~\eqref{manifold_subdivision}. The multiscale transform is defined by
\begin{equation}~\label{manifold_even_singular_multiscale}
    \boldsymbol{c}^{(\ell-1)} = \mathcal{Y}_{\boldsymbol{\gamma}}\boldsymbol{c}^{(\ell)}, \quad \boldsymbol{d}^{(\ell)} = \boldsymbol{c}^{(\ell)} \ominus \mathcal{T}_{\boldsymbol{\alpha}}\boldsymbol{c}^{(\ell-1)}, \quad \ell=1,\dots,J,
\end{equation}
where $\mathcal{Y}_{\boldsymbol{\gamma}}$ is the pseudo-reverse of $\mathcal{T}_{\boldsymbol{\alpha}}$ for some $\xi>0$. The inverse transform of~\eqref{manifold_even_singular_multiscale} is defined by iterating
\begin{equation}~\label{manifold_synthesis_iterations}
    \boldsymbol{c}^{(\ell)} = \mathcal{T}_{\boldsymbol{\alpha}}\boldsymbol{c}^{(\ell - 1)} \oplus \boldsymbol{d}^{(\ell)}
\end{equation}
for $\ell=1\dots,J$.
\end{defin}

The first difference between the manifold and linear versions of the transform lies in the detail coefficients. In particular, for the manifold-valued transform~\eqref{manifold_even_singular_multiscale}, the sequences $\boldsymbol{c}^{(\ell)}$, $\ell=1,\dots,J$ are $\mathcal{M}$-valued, while the detail coefficients $\boldsymbol{d}^{(\ell)}$ are elements in the tangent bundle $T\mathcal{M}=\bigcup_{p\in\mathcal{M}}T_p\mathcal{M}$ associated with $\mathcal{M}$.

To investigate the properties of the multiscale transform~\eqref{manifold_even_singular_multiscale}, it is useful to introduce the following notation. For two sequences $\boldsymbol{c}$ and $\boldsymbol{m}$ in $\mathcal{M}$ we define
\begin{equation*}
    \mu(\boldsymbol{c}, \boldsymbol{m})=\sup_{j\in\mathbb{Z}}\rho(c_j, m_j) \quad \text{and} \quad \Delta_\mathcal{M}\boldsymbol{c} = \sup_{j\in\mathbb{Z}}\rho(c_{j}, c_{j+1}).
\end{equation*}

The following lemma is a borrowed result from~\cite{wallner2011convergence}, see Lemma 3 therein, and it concerns Cartan-Hadamard manifolds, which are simply connected manifolds with non-positive sectional curvature. The lemma is tailored to our setting and will be useful in proving our theoretical results.

\begin{lemma}~\label{lem:manifold_refinement_distance}
    Let $\mathcal{M}$ be a Cartan-Hadamard manifold, and let $\boldsymbol{c}$ be an arbitrary $\mathcal{M}$-valued sequence. Let $\mathcal{T}_{\boldsymbol{\alpha}}$ be a subdivision scheme adapted to $\mathcal{M}$ via~\eqref{manifold_subdivision} with non-negative coefficients $\boldsymbol{\alpha}$. Then, for any $\xi>0$, we have
    \begin{equation}~\label{eqn:perturbated_RCoM}
        \mu(\mathcal{T}_{\boldsymbol{\alpha}}\boldsymbol{c}, \mathcal{T}_{\widetilde{\boldsymbol{\alpha}}}\boldsymbol{c})\leq \sigma^2 \|\boldsymbol{\alpha}-\widetilde{\boldsymbol{\alpha}}\|_1\Delta_{\mathcal{M}}\boldsymbol{c},
    \end{equation}
    where $\sigma$ is the support size of $\boldsymbol{\alpha}$.
\end{lemma}

Although the estimation~\eqref{eqn:perturbated_RCoM} is true for sequences in Cartan-Hadamard spaces, it also holds true for general manifolds under the assumption that the data points of $\boldsymbol{c}$ are sufficiently dense~\cite{wallner2005convergence, wallner2006smoothness}. We are now ready to present the analogue of Lemma~\ref{lem:details}.

\begin{lemma}~\label{lem:Manifold_details}
    Let $\boldsymbol{c}^{(J)}$ be a sequence in a Cartan-Hadamard manifold $\mathcal{M}$, and let $\mathcal{T}_{\boldsymbol{\alpha}}$ be a refinement operator, where $\boldsymbol{\alpha}$ has non-negative coefficients, with a pseudo-reverse $\mathcal{Y}_{\boldsymbol{\gamma}}$ for some $\xi>0$, that is, $\boldsymbol{\gamma}=(\widetilde{\boldsymbol{\alpha}}\downarrow2)^{-1}$. Denote by $\{\boldsymbol{c}^{(0)};\boldsymbol{d}^{(1)},\dots,\boldsymbol{d}^{(J)}\}$ the multiscale representation~\eqref{manifold_even_singular_multiscale} of $\boldsymbol{c}^{(J)}$ using $\mathcal{T}_{\boldsymbol{\alpha}}$ and $\mathcal{Y}_{\boldsymbol{\gamma}}$. Then,
    \begin{equation}~\label{eqn:manifold_details_difference}
        \|\boldsymbol{d}^{(\ell)}\downarrow 2\|_\infty \leq L_\mathcal{M}(\xi) \Delta_\mathcal{M}\boldsymbol{c}^{(\ell)},
    \end{equation}
    for any $\ell=1,\dots,J$ where $\|\boldsymbol{d}\|_{\infty}=\sup_{j\in\mathbb{Z}}\|d_j\|$ for any sequence $\boldsymbol{d}$ in the tangent bundle of $\mathcal{M}$, and
    \begin{equation}~\label{eqn:constant_L_M}
        L_\mathcal{M}(\xi) = F_{\mathcal{Y}}\sigma^2\|\boldsymbol{\alpha}-\widetilde{\boldsymbol{\alpha}}\|_1,
    \end{equation}
    for some constant $F_{\mathcal{Y}}>0$, where $\sigma$ is the support size of $\boldsymbol{\alpha}$.
\end{lemma}

\begin{proof}
    We first observe that because $\widetilde{\boldsymbol{\alpha}}$ and $\boldsymbol{\gamma}$ satisfy~\eqref{Z_transform}, the resulting detail coefficients $q^{(\ell)}_{2j}$ obtained using $\mathcal{T}_{\widetilde{\boldsymbol{\alpha}}}$ and $\mathcal{Y}_{\boldsymbol{\gamma}}$ in~\eqref{manifold_even_singular_multiscale}, are the zero elements in the respective tangent spaces of $\mathcal{M}$, for all $j\in\mathbb{Z}$ and $\ell=1,\dots,J$. That is, $\rho\big(c^{(\ell)}_{2j}, (\mathcal{T}_{\widetilde{\boldsymbol{\alpha}}}\boldsymbol{c}^{(\ell-1)})_{2j}\big)=0$. Furthermore, Proposition 5.4 of~\cite{mattar2023pyramid} shows that $\mathcal{Y}_{\boldsymbol{\gamma}}$ is \emph{decimation-safe}, that is, there exists a constant $F_{\mathcal{Y}}>0$ such that $\Delta_{\mathcal{M}}(\mathcal{Y}_{\boldsymbol{\gamma}}\boldsymbol{c})\leq F_{\mathcal{Y}}\Delta_{\mathcal{M}}\boldsymbol{c}$ for any sequence $\boldsymbol{c}$ in $\mathcal{M}$. Now, observe that
    \begin{align*}
        \|d^{(\ell)}_{2j}\| = \rho\big(c^{(\ell)}_{2j}, (\mathcal{T}_{\boldsymbol{\alpha}}\boldsymbol{c}^{(\ell-1)})_{2j}\big) & \leq \rho\big(c^{(\ell)}_{2j}, (\mathcal{T}_{\widetilde{\boldsymbol{\alpha}}}\boldsymbol{c}^{(\ell-1)})_{2j}\big) + \rho\big((\mathcal{T}_{\widetilde{\boldsymbol{\alpha}}}\boldsymbol{c}^{(\ell-1)})_{2j}, (\mathcal{T}_{\boldsymbol{\alpha}}\boldsymbol{c}^{(\ell-1)})_{2j}\big).
    \end{align*}
    Now, since the first distance vanishes, we get
    \begin{align*}
        \|\boldsymbol{d}^{(\ell)}\downarrow2\|_\infty &\leq \sup_{j\in\mathbb{Z}}\rho\big((\mathcal{T}_{\widetilde{\boldsymbol{\alpha}}}\boldsymbol{c}^{(\ell-1)})_{j}, (\mathcal{T}_{\boldsymbol{\alpha}}\boldsymbol{c}^{(\ell-1)})_{j}\big) = \mu\big(\mathcal{T}_{\widetilde{\boldsymbol{\alpha}}}\boldsymbol{c}^{(\ell-1)}, \mathcal{T}_{\boldsymbol{\alpha}}\boldsymbol{c}^{(\ell-1)}\big) \\
        & \leq \sigma^2\|\boldsymbol{\alpha} - \widetilde{\boldsymbol{\alpha}}\|_1\Delta_{\mathcal{M}}(\mathcal{Y}_{\boldsymbol{\gamma}}\boldsymbol{c}^{(\ell)}) \leq F_{\mathcal{Y}} \sigma^2\|\boldsymbol{\alpha} - \widetilde{\boldsymbol{\alpha}}\|_1\Delta_{\mathcal{M}}\boldsymbol{c}^{(\ell)},
    \end{align*}
    as required, where the second line is based on Lemma~\ref{lem:manifold_refinement_distance}.
\end{proof}

Conclusions similar to those of Lemma~\ref{lem:details} can be drawn from Lemma~\ref{lem:Manifold_details} for the manifold-valued transform~\eqref{manifold_even_singular_multiscale}. In particular, if the refinement $\mathcal{T}_{\boldsymbol{\alpha}}$ is reversible, then $\|\boldsymbol{\alpha}-\widetilde{\boldsymbol{\alpha}}\|_1=0$ for all $\xi>0$, the constant $L_\mathcal{M}(\xi)$ of~\eqref{eqn:constant_L_M} vanishes as a result, and hence we have zero even detail coefficients. Otherwise, the maximal norm of the even detail coefficients $d^{(\ell)}_{2j}$, $j\in\mathbb{Z}$ does not only grow with respect to the perturbation $\|\boldsymbol{\alpha}-\widetilde{\boldsymbol{\alpha}}\|_1$, but also depends on the nature of the analyzed sequence. In the following corollary, we present the analogue of Theorem~\ref{thm:details_difference}.

\begin{coro}
    Under the assumption of Lemma~\ref{lem:Manifold_details}, if the sequence $\boldsymbol{c}^{(J)}$ is assumed to be sampled from a differentiable curve $\Gamma\subset\mathcal{M}$ over the arc-length parametrization grid $2^{-J}\mathbb{Z}$, then the term $\Delta_\mathcal{M}\boldsymbol{c}^{(\ell)}$ of~\eqref{eqn:manifold_details_difference} satisfies
    \begin{equation*}
        \Delta_{\mathcal{M}}\boldsymbol{c}^{(\ell)}\leq P_{\mathcal{M}}^J(\xi)\|\nabla\Gamma\|_\infty(2P_{\mathcal{M}}(\xi))^{-\ell},
    \end{equation*}
    for all $\ell=1,\dots,J$ and some $P_{\mathcal{M}}(\xi)>1$, where $\|\nabla\Gamma\|_\infty=\sup_{t\in\mathbb{R}}\|\nabla\Gamma(t)\|$. This fact appears as Lemma 5.8 in~\cite{mattar2023pyramid}. Consequently, the even detail coefficients must decay geometrically with respect to the scale $\ell$ by the factor $2P_{\mathcal{M}}(\xi)$.
\end{coro}

We next analyze the synthesis error in a similar manner to Theorem~\ref{thm:synthesis_theorem}. To this end, we define the \emph{stability} of refinement operators. A refinement operator $\mathcal{T}_{\boldsymbol{\alpha}}$ of~\eqref{manifold_subdivision} is called \emph{stable} if there exists a constant $S_{\mathcal{T}}>0$ such that
\begin{equation}~\label{eqn:stability}
    \mu (\mathcal{T}_{\boldsymbol{\alpha}}\boldsymbol{c}, \mathcal{T}_{\boldsymbol{\alpha}}\boldsymbol{m}) \leq S_{\mathcal{T}} \mu(\boldsymbol{c},\boldsymbol{m}),
\end{equation}
for all sequences $\boldsymbol{c}$ and $\boldsymbol{m}$. The stability condition has been studied, for example, in~\cite{grohs2010stability}.

It turns out that an analogue of Theorem~\ref{thm:synthesis_theorem} can be obtained intrinsically, when the curvature of the manifold is bounded. Next, we present such a result assuming $\mathcal{M}$ is complete, without boundary, with non-negative sectional curvature. For that, we recall two classical theorems: the first and second Rauch comparison theorems. For more details see~\cite[Chapter 3]{gromoll2009metric} and references therein.

Let $p_j \in \mathcal{M}$, $j=1,2$ be two points and $v_j \in T_{p_j}\mathcal{M}$ their vectors in the tangent spaces such that $\| v_1 \|=\| v_2 \|$ and the value is smaller than the injectivity radius of $\mathcal{M}$. Let $ G(p_1,p_2)$ be the geodesic line connecting $p_1$ and $p_2$ and $\pg_{p_2}(v_1) \in T_{p_2}\mathcal{M}$ be the parallel transport of $v_1$ along $G(p_1,p_2)$ to $T_{p_2}\mathcal{M}$. Then, the first Rauch theorem suggests that
\begin{equation} \label{eqn:rauch1}
\rho \big(p_2 \oplus v_2, p_2 \oplus \pg_{p_2}(v_1)\big)  \le \|v_2- \pg_{p_2}(v_1) \|.
\end{equation}
Furthermore, the second Rauch theorem implies that 
\begin{equation} \label{eqn:rauch2}
    \rho \big(p_1 \oplus v_1, p_2 \oplus \pg_{p_2}(v_1)\big)  \le \rho( p_1, p_2).
\end{equation}

We are now ready to study the reconstruction error that stems from the synthesis~\eqref{manifold_synthesis_iterations}, after setting half of the detail coefficients to zero, in an analogous manner to Theorem~\ref{thm:synthesis_theorem}.

\begin{theorem}~\label{ManifoldStabilityTHM}
    Let $\boldsymbol{c}^{(J)}$ be an $\mathcal{M}$-valued sequence where $\mathcal{M}$ is a complete, without boundary, with non-negative sectional curvature,  and let $\mathcal{T}_{\boldsymbol{\alpha}}$ be a refinement operator~\eqref{manifold_subdivision} with a pseudo-reverse $\mathcal{Y}_{\boldsymbol{\gamma}}$ for some $\xi>0$. Denote by $\{\boldsymbol{c}^{(0)};\boldsymbol{d}^{(1)},\dots,\boldsymbol{d}^{(J)}\}$ the multiscale representation~\eqref{manifold_even_singular_multiscale} of $\boldsymbol{c}^{(J)}$ using $\mathcal{T}_{\boldsymbol{\alpha}}$ and $ \mathcal{Y}_{\boldsymbol{\gamma}}$. Then, there exists $C_{\mathcal{M}}>0$ such that
    \begin{align}~\label{eqn:manifold_stability}
            \mu\big(\boldsymbol{c}^{(J)}, \boldsymbol{\zeta}^{(J)}\big) \leq C_{\mathcal{M}}\sum_{\ell=1}^{J}\|\boldsymbol{d}^{(\ell)}\downarrow 2\|_\infty,
    \end{align}
    where $\boldsymbol{\zeta}^{(J)}$ is the synthesized sequence after setting the even detail coefficients of the multiscale representation to zero. In other words, the reconstruction error is proportional to the cumulative errors studied in~\eqref{eqn:manifold_details_difference}.
\end{theorem}

\begin{proof}
    The proof is a generalization to the proof of Theorem~\ref{thm:synthesis_theorem}. We consider the pyramid $\{\boldsymbol{c}^{(0)};\boldsymbol{q}^{(1)},\dots, \boldsymbol{q}^{(J)}\}$ where the details are given by
    \begin{equation*}
        q_{j}^{(\ell)} = \begin{cases}
            d^{(\ell)}_j, & j=2n+1, \\
            0, & j=2n,
        \end{cases} \quad \quad\ell=1,\dots, J, \quad j\in\mathbb{Z}.
    \end{equation*}
    We aim to bound the distance between $\boldsymbol{c}^{(J)}$ and the reconstructed sequence $\boldsymbol{\zeta}^{(J)}$ obtained by synthesizing the pyramid with the details $\boldsymbol{q}^{(\ell)}$, $\ell=1,\dots, J$. To this end, let us denote by $\widehat{\boldsymbol{q}}^{(\ell)}=\text{PG}_{\mathcal{T}_{\boldsymbol{\alpha}}\boldsymbol{c}^{(\ell-1)}}(\boldsymbol{q}^{(\ell)})$, and observe that
    \begin{align*}
        \mu\big(\boldsymbol{c}^{(\ell)}, \boldsymbol{\zeta}^{(\ell)}\big) & = \mu\big(\mathcal{T}_{\boldsymbol{\alpha}}\boldsymbol{c}^{(\ell-1)}\oplus\boldsymbol{d}^{(\ell)}, \mathcal{T}_{\boldsymbol{\alpha}}\boldsymbol{\zeta}^{(\ell-1)}\oplus\boldsymbol{q}^{(\ell)}\big) \\
        & \leq \mu\big(\mathcal{T}_{\boldsymbol{\alpha}}\boldsymbol{c}^{(\ell-1)}\oplus\boldsymbol{d}^{(\ell)}, \mathcal{T}_{\boldsymbol{\alpha}}\boldsymbol{c}^{(\ell-1)}\oplus\widehat{\boldsymbol{q}}^{(\ell)}\big) + \mu\big( \mathcal{T}_{\boldsymbol{\alpha}}\boldsymbol{c}^{(\ell-1)}\oplus\widehat{\boldsymbol{q}}^{(\ell)}, \mathcal{T}_{\boldsymbol{\alpha}}\boldsymbol{\zeta}^{(\ell-1)}\oplus\boldsymbol{q}^{(\ell)}\big) \\
        & \leq \|\boldsymbol{d}^{(\ell)}-\widehat{\boldsymbol{q}}^{(\ell)}\|_\infty + \mu\big( \mathcal{T}_{\boldsymbol{\alpha}}\boldsymbol{c}^{(\ell-1)}, \mathcal{T}_{\boldsymbol{\alpha}}\boldsymbol{\zeta}^{(\ell-1)}\big) \\
        & \leq \|\boldsymbol{d}^{(\ell)}\downarrow 2\|_\infty + S_{\mathcal{T}} \mu\big(\boldsymbol{c}^{(\ell-1)}, \boldsymbol{\zeta}^{(\ell-1)}\big).
    \end{align*}
    The first inequality is due to the triangular inequality of $\mu$, while the second inequality is due to~\eqref{eqn:rauch1} and~\eqref{eqn:rauch2}. Lastly, the third inequality is due to the stability of $\mathcal{T}_{\boldsymbol{\alpha}}$, see~\eqref{eqn:stability}, and the fact that parallel transport preserves the lengths of the transported vectors along their corresponding geodesics.
    Now, repeating the above inequalities starting from $\ell=J$, and bearing in mind that $\boldsymbol{c}^{(0)}=\boldsymbol{\zeta}^{(0)}$ gives
    \begin{equation*}
        \mu\big(\boldsymbol{c}^{(J)}, \boldsymbol{\zeta}^{(J)}\big)\leq \sum_{\ell=1}^J S_{\mathcal{T}}^{J-\ell}\|\boldsymbol{d}^{(\ell)}\downarrow 2\|_\infty.
    \end{equation*}
    The required result in then obtained by taking $C_{\mathcal{M}}=S_{\mathcal{T}}^J$ if $S_{\mathcal{T}}>1$ or otherwise $C_{\mathcal{M}}=1$.
\end{proof}

%%%%% Numerical Examples %%%%%

\section{Applications and numerical examples}\label{sec:numerical_examples}

Next, we illustrate different applications of our multiscale transforms. All results are reproducible via a package of Python code available online at \href{https://github.com/WaelMattar/Pseudo-reversing.git}{https://github.com/WaelMattar/Pseudo-reversing.git}. We start with numerical illustrations of pseudo-reversing subdivision schemes as refinement operators, in the linear setting, as Figure~\eqref{diag:pseudo_reversing_operators} shows.

\subsection{Pseudo-reversing subdivision schemes}
Let $\mathcal{S}_{\boldsymbol{\alpha}}$ be the subdivision scheme~\eqref{linear_subdivision_scheme} given with the mask
\begin{equation}~\label{LS_scheme_n_2_d_1}
    \boldsymbol{\alpha} = \bigg[\displaystyle \frac{1}{4},\; \frac{1}{3},\; \frac{1}{4},\; \frac{1}{3},\; \frac{1}{4},\; \frac{1}{3},\; \frac{1}{4}\bigg] \quad \text{supported on} \quad [-3,-2,-1,0,1,2,3].
\end{equation}
This subdivision is a member of a broader family of least-squares schemes~\cite{dyn2015univariate}, and its corresponding symbol $(\boldsymbol{\alpha}\downarrow 2)(z)$ of~\eqref{Z_transform} is given by
\begin{equation}~\label{even_symbol_of_n_2_d_1}
    (\boldsymbol{\alpha}\downarrow 2)(z) = \frac{1}{3z} + \frac{1}{3} + \frac{z}{3}.
\end{equation}
The polynomial $z(\boldsymbol{\alpha}\downarrow 2)(z)$ appears in Example~\ref{example_2} and it possesses two zeros on the unit circle; $-1/2 \pm i \sqrt{3}/2$. Figure~\ref{fig:pseudo-reversing} demonstrates the effect of pseudo-reversing the symbol $\boldsymbol{\alpha}$ with different parameters of $\xi$, highlighting the resulting tradeoff; when $\xi$ values are large, we obtain better-decaying decimation coefficients but also larger deviations from the original sequence. On the contrary, when we perturb only slightly with small $\xi$ values, the decimation, comprised of the reversed sequence, grows significantly, making its practical use less feasible.

\begin{figure}[!htbp]
    \begin{subfigure}[b]{0.40\textwidth}  
        \includegraphics[width=\textwidth]{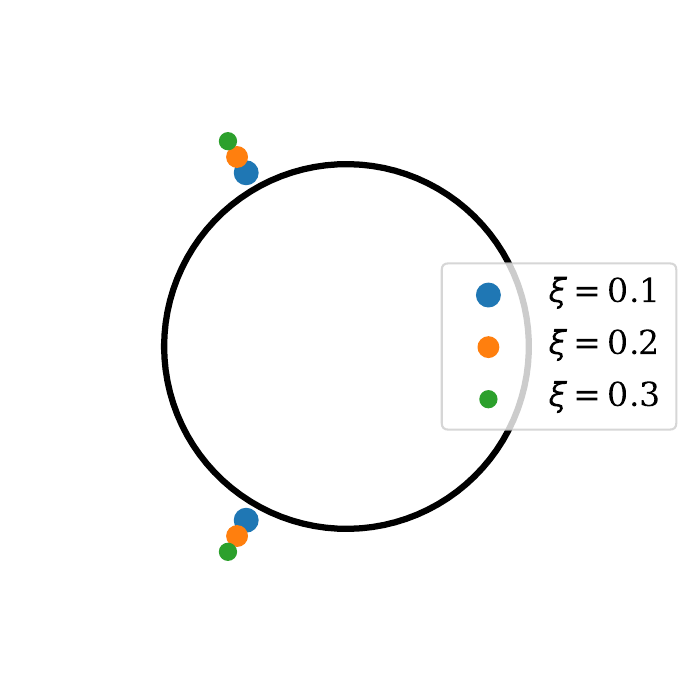}
        \caption{Displacements of zeros relative to $\mathbb{T}$.}
        \label{fig:root_displacement}
    \end{subfigure}
    \hfill
    \begin{subfigure}[b]{0.48\textwidth}   
        \includegraphics[width=\textwidth]{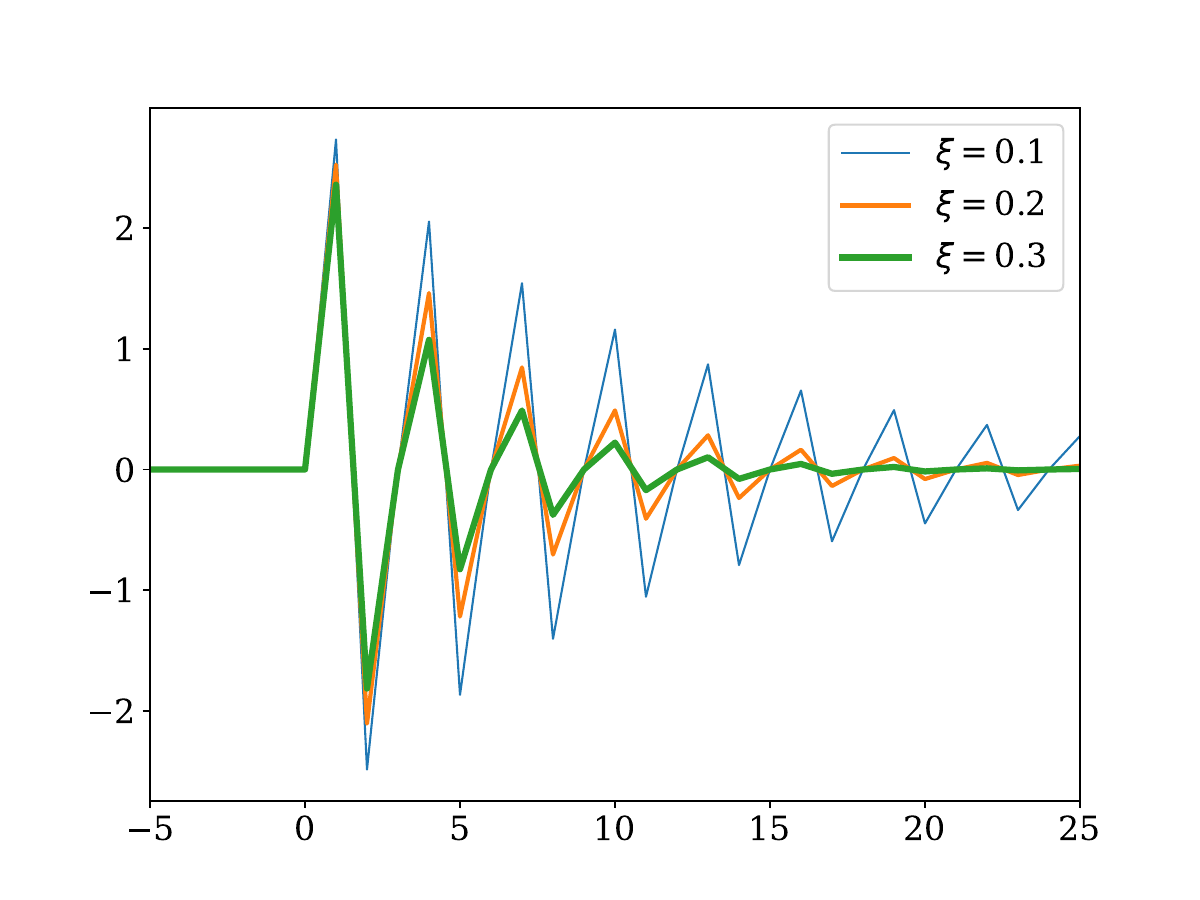}
        \caption{Decimation coefficients.}   
        \label{fig:decimation_coefficients}
    \end{subfigure}
    \vskip\baselineskip
    \begin{subfigure}[b]{0.48\textwidth}
        \includegraphics[width=\textwidth]{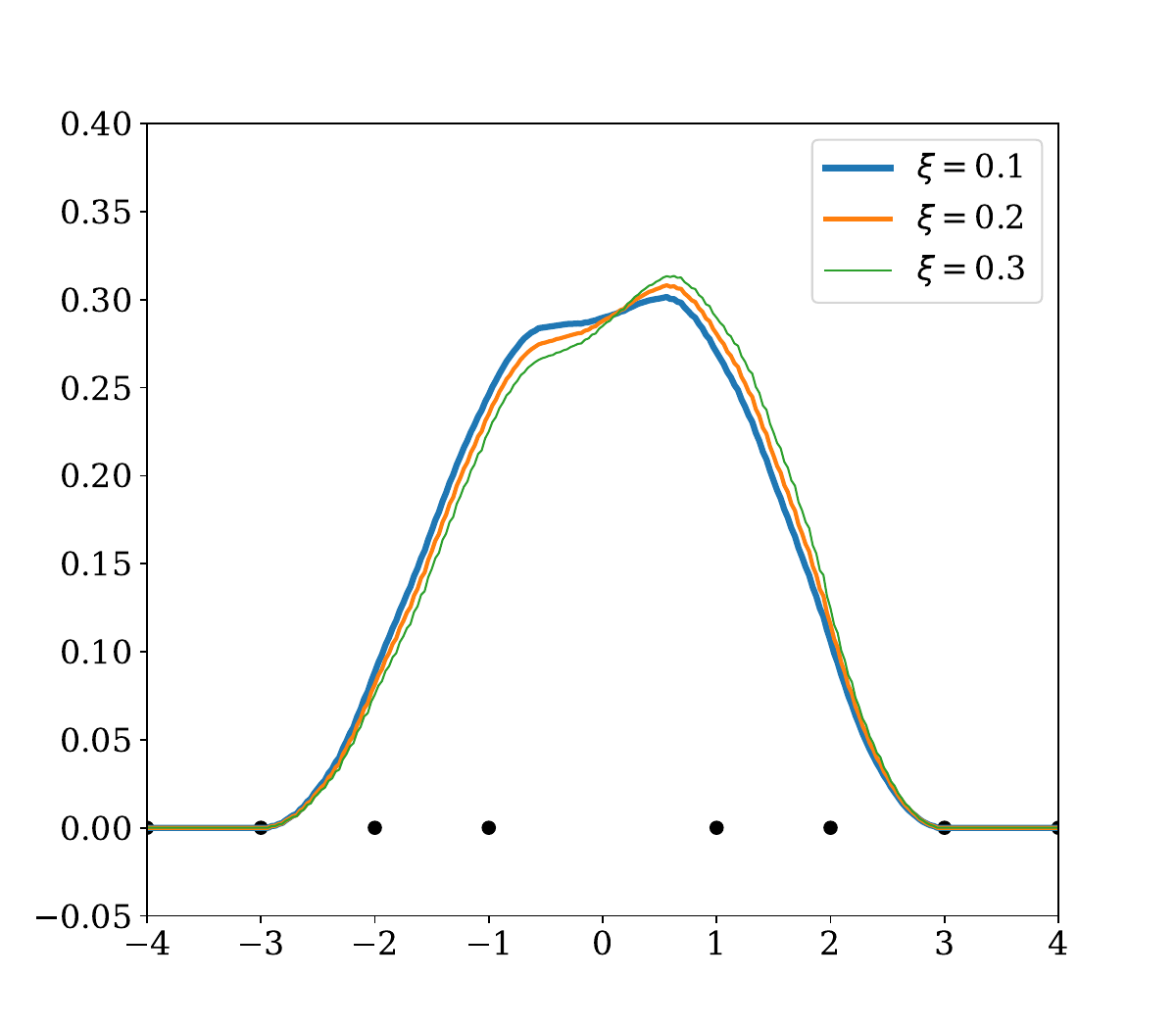}
        \caption{Basic limit functions of $\mathcal{S}_{\widetilde{\boldsymbol{\alpha}}}$.}
        \label{fig:limit_functions}
    \end{subfigure}
    \hfill
    \begin{subfigure}[b]{0.48\textwidth}
        \includegraphics[width=\textwidth]{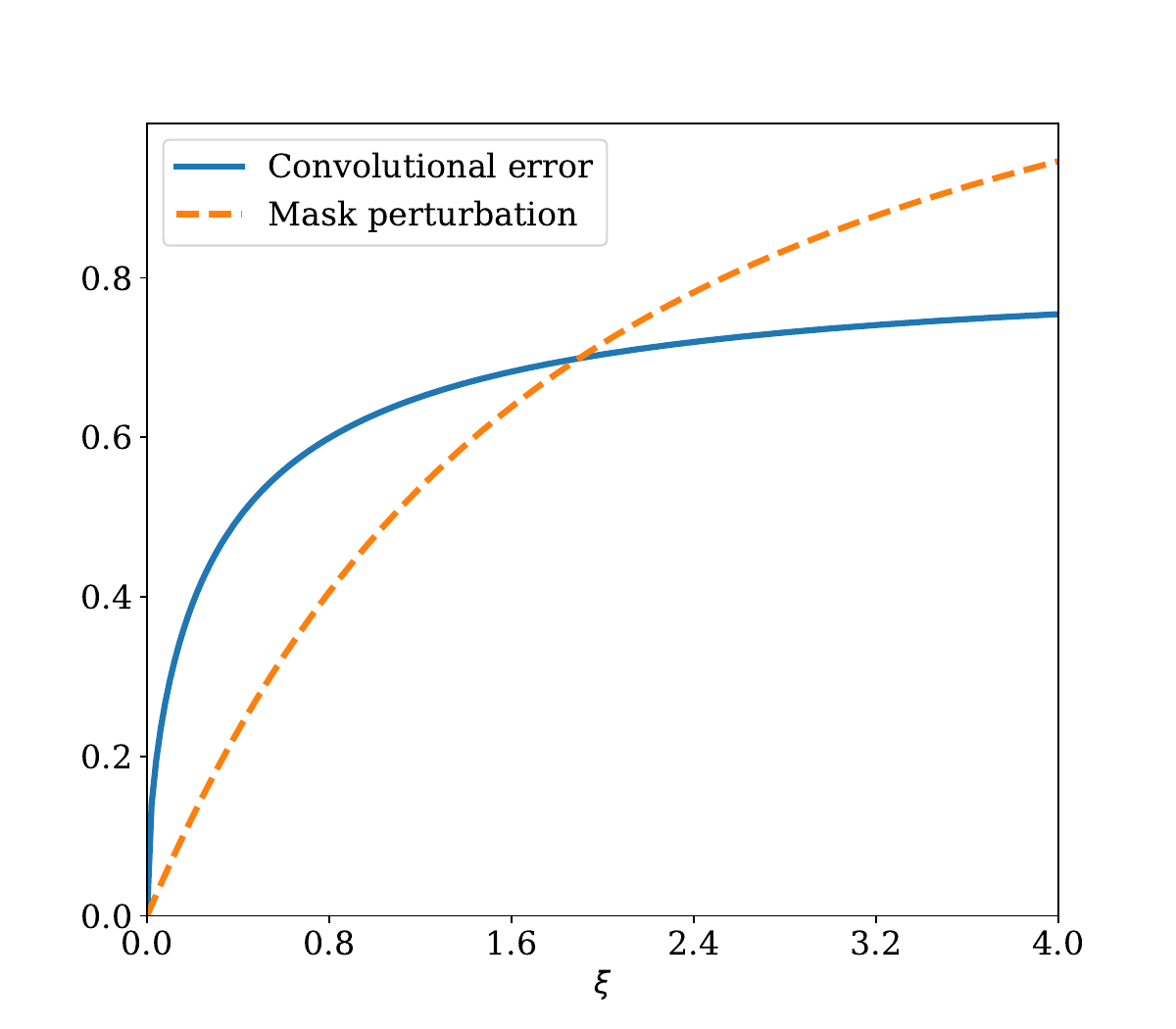}
        \caption{Convolutional error.}
        \label{fig:convolutional_error}
    \end{subfigure}
    \caption{Pseudo-reversing the subdivision scheme~\eqref{LS_scheme_n_2_d_1}. Figure~\eqref{fig:root_displacement} illustrates pushing the complex-valued zeros of~\eqref{even_symbol_of_n_2_d_1} with a parameter $\xi$. The pseudo-reverse coefficients of $\mathcal{D}_{\boldsymbol{\gamma}}$ are illustrated in Figure~\eqref{fig:decimation_coefficients}. The basic limit functions, that are $\mathcal{S}_{\widetilde{\boldsymbol{\alpha}}}^\infty \boldsymbol{\delta}$ where $\boldsymbol{\delta}$ is the Kronecker delta sequence, are depicted in Figure~\eqref{fig:limit_functions}. These functions are vital for studying the subdivision scheme $\mathcal{S}_{\widetilde{\boldsymbol{\alpha}}}$ as they reveal many of its properties. Figure~\eqref{fig:convolutional_error} illustrates the convolutional error of~\eqref{convolutional_equation}, $\|\boldsymbol{\delta} - \boldsymbol{\gamma}*(\boldsymbol{\alpha}\downarrow 2)\|_2$ as well as the norm of the mask perturbation $\|\boldsymbol{\alpha} - \widetilde{\boldsymbol{\alpha}}\|_1$.}
    \label{fig:pseudo-reversing}
\end{figure}

To shed more light on pseudo-reversing, Table~\ref{table:LS_kappa} shows the reversibility condition number $\kappa$ in~\eqref{eqn:condition_number} of $\widetilde{\boldsymbol{\alpha}}$ for different values of $\xi$, and for the same basic refinement. This table clearly shows the inverse correlation between $\xi$ and $\kappa$ as the theory suggests in Corollary~\ref{coro:kappa_of_pseudo}.

\begin{table}[H]
\begin{center}
    \begin{tabular}{| c || c | c | c | c | c | c | c | c | c | c | c | c | c | c | c |} 
     \hline
     $\xi$ & 0 & 0.1 & 0.2 & 0.3 & 0.4 & 0.5 & 0.6 & 0.7 & 0.8 & 0.9 & 1 & 1.1 & 1.2 \\
     \hline
     $\|\boldsymbol{\alpha}-\widetilde{\boldsymbol{\alpha}}\|_1$ & 0 & 0.06 & 0.12 & 0.18 & 0.23 &  0.28 & 0.32 & 0.36 &  0.40 & 0.44 & 0.47 & 0.50 & 0.53 \\
     \hline
     $\kappa(\widetilde{\boldsymbol{\alpha}})$ & $\infty$ & 18.19 & 9.54 & 6.67 & 5.24 & 4.38 & 3.81 & 3.41 & 3.11 & 2.88 & 2.69 & 2.54 & 2.41 \\
     \hline
    \end{tabular}
\end{center}
    \caption{The mask perturbation and the reversibility condition number of the approximating mask $\widetilde{\boldsymbol{\alpha}}$ against different parameters $\xi$.}
    \label{table:LS_kappa}
\end{table}

As mentioned in Remark~\ref{remark:Bspline}, the notion of pseudo-reversing can be relaxed and applied to refinements with bad reversibility condition numbers. That is, roughly speaking, schemes with zeros close to the unit circle. In other words, pseudo-reversing allows us to enforce a better, more practical reversibility. Technically, we can do this by pushing the zeros of $(\boldsymbol{\alpha}\downarrow 2)(z)$, which have moduli \emph{greater} than 1, with the factor $\xi$ as similar to~\eqref{eqn:pseudo_reverse}.

Figure~\ref{fig:Bspline} illustrates the zeros of $(\boldsymbol{\alpha}\downarrow 2)(z)$ in~\eqref{Z_transform} and the coefficients of the solution $\boldsymbol{\gamma}$, where $\boldsymbol{\alpha}$ is the mask of a high-order B-spline subdivision scheme. In addition, Table~\ref{table:Bspline_kappa} shows the reversibility condition number $\kappa$ of the reversible schemes, while Table~\ref{table:Bspline_pseudo_reversing_kappa} illustrates how pseudo-reversing imposes a better condition number on the B-spline subdivision scheme of order 6.

\begin{figure}[H]
    \begin{subfigure}[c]{0.48\textwidth}
        \includegraphics[width=\textwidth]{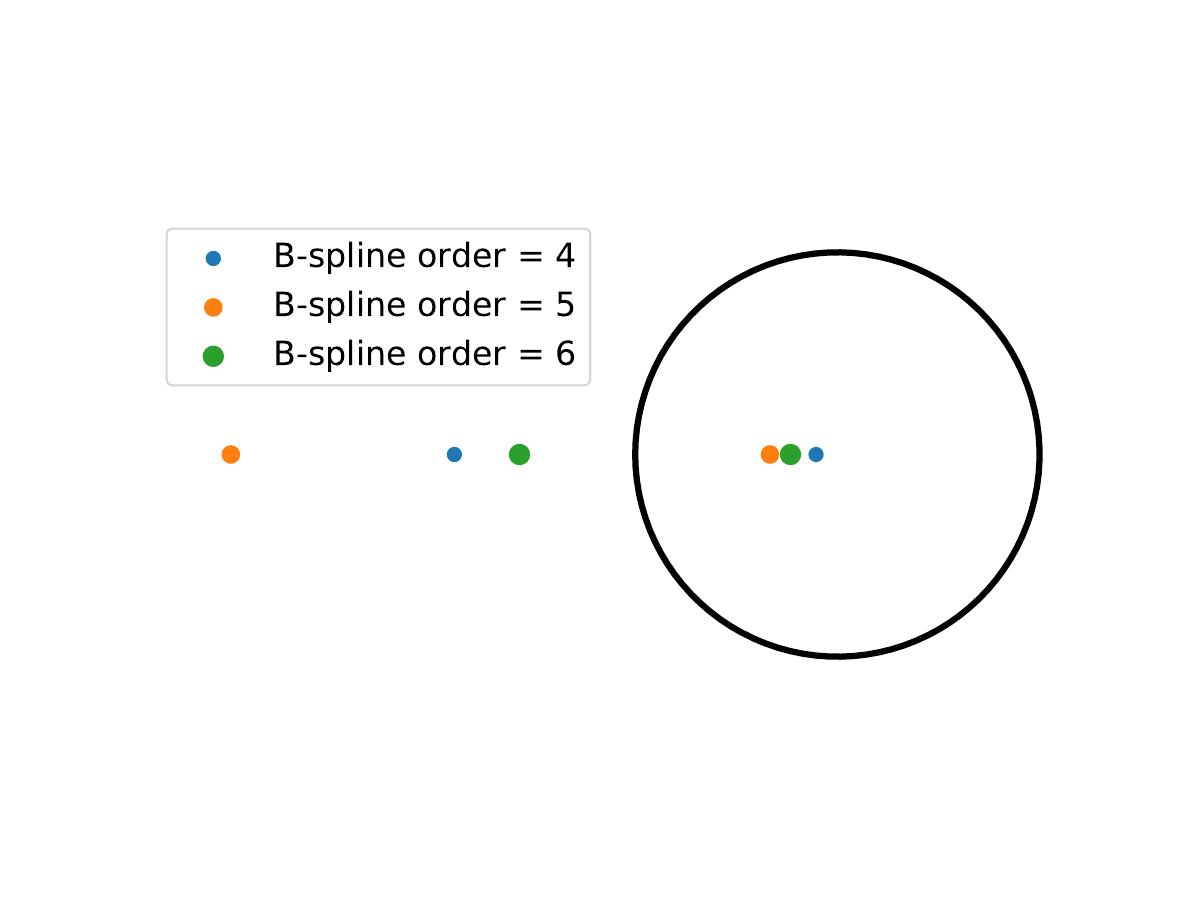}
        \caption{Zeros of high-order B-spline schemes.}
        \label{fig:Bspline_roots}
    \end{subfigure}
    \hfill
    \begin{subfigure}[c]{0.48\textwidth}  
        \includegraphics[width=\textwidth]{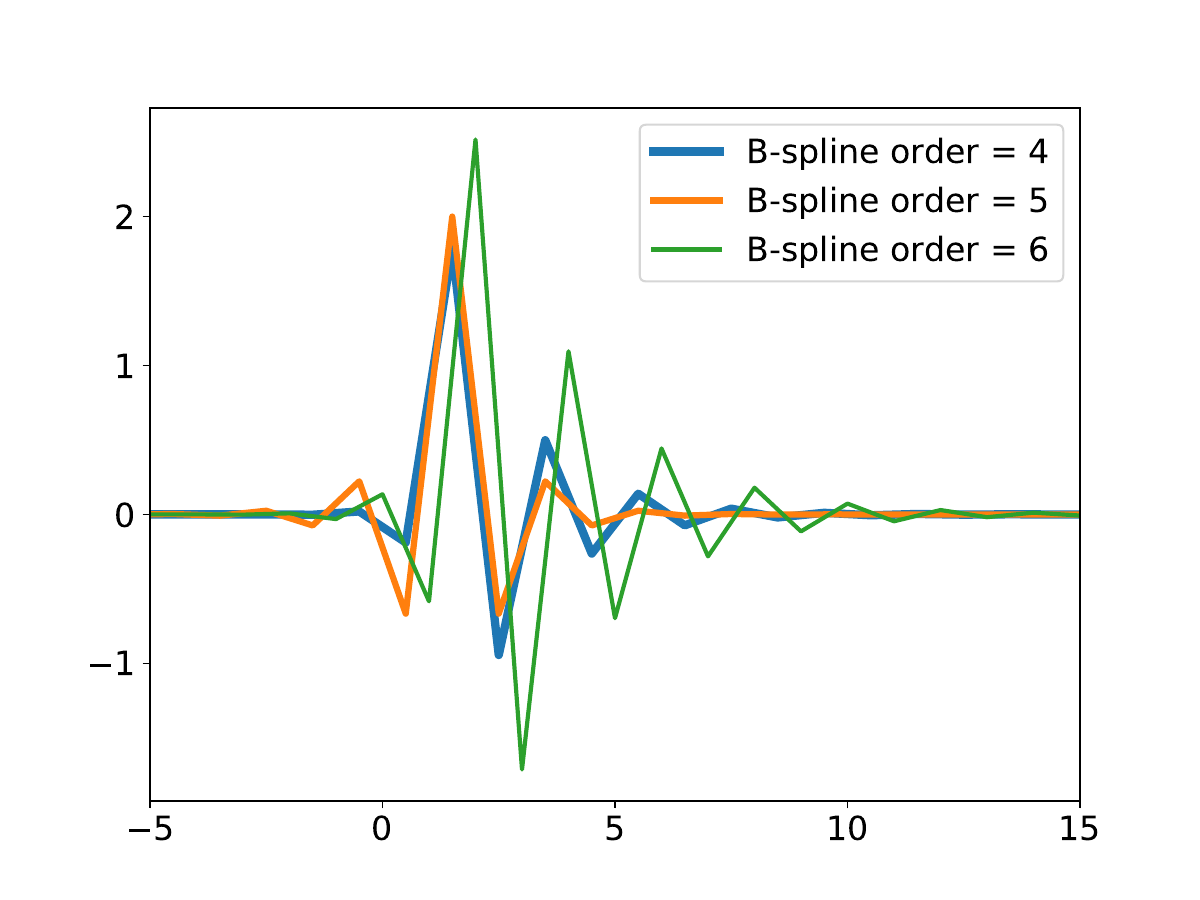}
        \caption{Decimation coefficients.}
        \label{fig:Bspline_decimation}
    \end{subfigure}
    \caption{Zeros and decimation coefficients of high-order B-spline subdivision schemes. On the left, part of the zeros of $(\boldsymbol{\alpha}\downarrow2)(z)$ in~\eqref{Z_transform} corresponding to the B-spline schemes of different orders, relative to $\mathbb{T}$. On the right, a depiction of the corresponding decimation coefficients. The reversibility condition number of these schemes appears in Table~\ref{table:Bspline_kappa}.}
    \label{fig:Bspline}
\end{figure}

\begin{table}[H]
\begin{center}
    \begin{tabular}{| c || c | c | c | c | c | c | c | c |} 
     \hline
     B-spline order & 2 & 3 & 4 & 5 & 6 & 7 \\
     \hline
     $\kappa(\boldsymbol{\alpha})$ & 2 & 2 & 4 & 4 & 8 & 8 \\
     \hline
    \end{tabular}
\end{center}
    \caption{The reversibility condition number~\eqref{eqn:condition_number} of the B-spline subdivision schemes. Note how the reversibility weakens as the order of the B-spline increases. Interestingly, the pattern seems to follow the rule $\kappa(\boldsymbol{\alpha}\downarrow2)=2^{\lfloor n/2\rfloor}$ where $n$ is the B-spline order and $\left\lfloor\cdot\right\rfloor$ is the floor function. An explicit formula for $(\boldsymbol{\alpha}\downarrow 2)(z)$ of order $n$ is given by $2^{-n}\sum_{j=0}^{\lfloor(n+1)/2\rfloor}\binom{n+1}{2j} z^j$.}
    \label{table:Bspline_kappa}
\end{table}

\begin{table}[H]
\begin{center}
    \begin{tabular}{| c || c | c | c | c | c | c | c | c | c | c | c | c | c | c | c |} 
     \hline
     $\xi$ & 0 & 0.1 & 0.2 & 0.3 & 0.4 & 0.5 & 0.6 & 0.7 & 0.8 & 0.9 & 1 & 1.1 & 1.2 \\
     \hline
     $\kappa(\widetilde{\boldsymbol{\alpha}})$ & 8 & 6.59 & 5.69 & 5.06 & 4.60 & 4.25 & 3.97 & 3.74 & 3.55 & 3.39 & 3.26 & 3.14 & 3.04 \\
     \hline
    \end{tabular}
\end{center}
    \caption{Improving the reversibility condition number~\eqref{eqn:condition_number} of the B-spline subdivision scheme of order 6 with different parameters $\xi$. Only the zeros of $(\boldsymbol{\alpha}\downarrow 2)(z)$ with moduli greater than 1 were displaced with the parameter $\xi$. Note how the reversibility improves as $\xi$ increases.}
    \label{table:Bspline_pseudo_reversing_kappa}
\end{table}

\subsection{Multiscaling in the scalar-valued setting}~\label{sec:numerical_linear_multiscaling}

Here we numerically illustrate the linear transform~\eqref{linear_multiscale_analysis} and test the synthesis result appearing in Theorem~\ref{thm:synthesis_theorem}. For this sake, we let $\boldsymbol{c}^{(6)}$ be our analyzed sequence, sampled from a test function $f(x):[0,10]\mapsto\mathbb{R}$ over the equispaced bounded grid $x\in2^{-6}\mathbb{Z}\cap[0, 10]$. We take $f$ to be the real part of the standard Morlet wavelet, centered at $x=5$.

We implement the pyramid transform~\eqref{linear_multiscale_analysis} based on the subdivision scheme~\eqref{LS_scheme_n_2_d_1} and its pseudo-reverse $\mathcal{D}_{\boldsymbol{\gamma}}$ for $\xi=1.4$. We remind the reader that $\boldsymbol{\gamma}$ is infinitely supported, see Proposition~\ref{prop:decay_of_reverse}. However, in practice, we first truncate the support with two endpoints that cover the indices where its values are significant, making its support finite. Then, we normalize it in the sense of dividing its elements by its sum. In fact, this method is guaranteed to produce controllable errors, depending on the size of the truncation~\cite{mattar2023pyramid}.

We decompose the sequence $\boldsymbol{c}^{(6)}$ by iterating~\eqref{linear_multiscale_analysis} $m$ times, for $m=1,2,\dots,6$, to obtain the pyramid $\{ \boldsymbol{c}^{(6-m)}; \boldsymbol{d}^{(5-m)}, \dots, \boldsymbol{d}^{(6)}\}$. All detail coefficients on the even indices are then set to zero, yielding a sparser pyramid. Then, we reconstruct using the inverse pyramid transform. Theorem~\ref{thm:synthesis_theorem} suggests that the reconstruction error, measured by the infinity norm between the analyzed $\boldsymbol{c}^{(6)}$ and the reconstructed $\boldsymbol{\zeta}^{(6)}$ grows as $m$ increases. The reason being that the more we decompose into coarse scales, the loss of information resulting from setting the even detail coefficients to zero becomes more significant, and hence the upper bound~\eqref{synthesis_bound} increases, allowing more space for the error $\|\boldsymbol{c}^{(6)}-\boldsymbol{\zeta}^{(6)}\|_\infty$.

The upper row of Figure~\ref{fig:linear_example} demonstrates the analyzed sequence $\boldsymbol{c}^{(6)}$ next to its $4$ layers ($m=4$) of detail coefficients. Note how the property of having small detail coefficients on the even indices is more pronounced on fine scales, and is violated on coarse scales. This phenomenon is explained via Theorem~\ref{thm:details_difference}.

To further illustrate the theoretical results, we contaminate the sequence $\boldsymbol{c}^{(6)}$ with an additive noise sampled from a normal Gaussian distribution with zero mean and standard deviation of $1\%$. The lower row of Figure~\ref{fig:linear_example} shows the noisy sequence next to two layers of detail coefficients $(m=4)$. The reconstruction error grows with the number of decompositions, as affirmed by Theorem~\ref{thm:synthesis_theorem}. Moreover, notice that there is no clear structured pattern in the detail coefficients and that its even details are significant, in contrast to Figure~\ref{fig:linear_pyramid}. This is explained by Theorem~\ref{thm:details_difference}.

Table~\ref{table:synthesis_error} shows the synthesis error with respect to the decomposition layer $m$ where $m$ runs from $1$ to $6$, both for the smooth and the noisy curves appearing in Figure~\ref{fig:linear_example}.

We finally remark here that similar results are obtained for different values of $\xi>0$, but we picked the value $\xi=1.4$ because it yielded a good reversibility condition number $\kappa$ that was suitable for the truncation size of the decimation operator $\mathcal{D}_{\boldsymbol{\gamma}}$. In general, one should pick values of $\xi$ such that the resulting even detail coefficients are as small as possible.

\begin{figure}[H]
    \begin{subfigure}[b]{0.49\textwidth}
        \includegraphics[width=\textwidth]{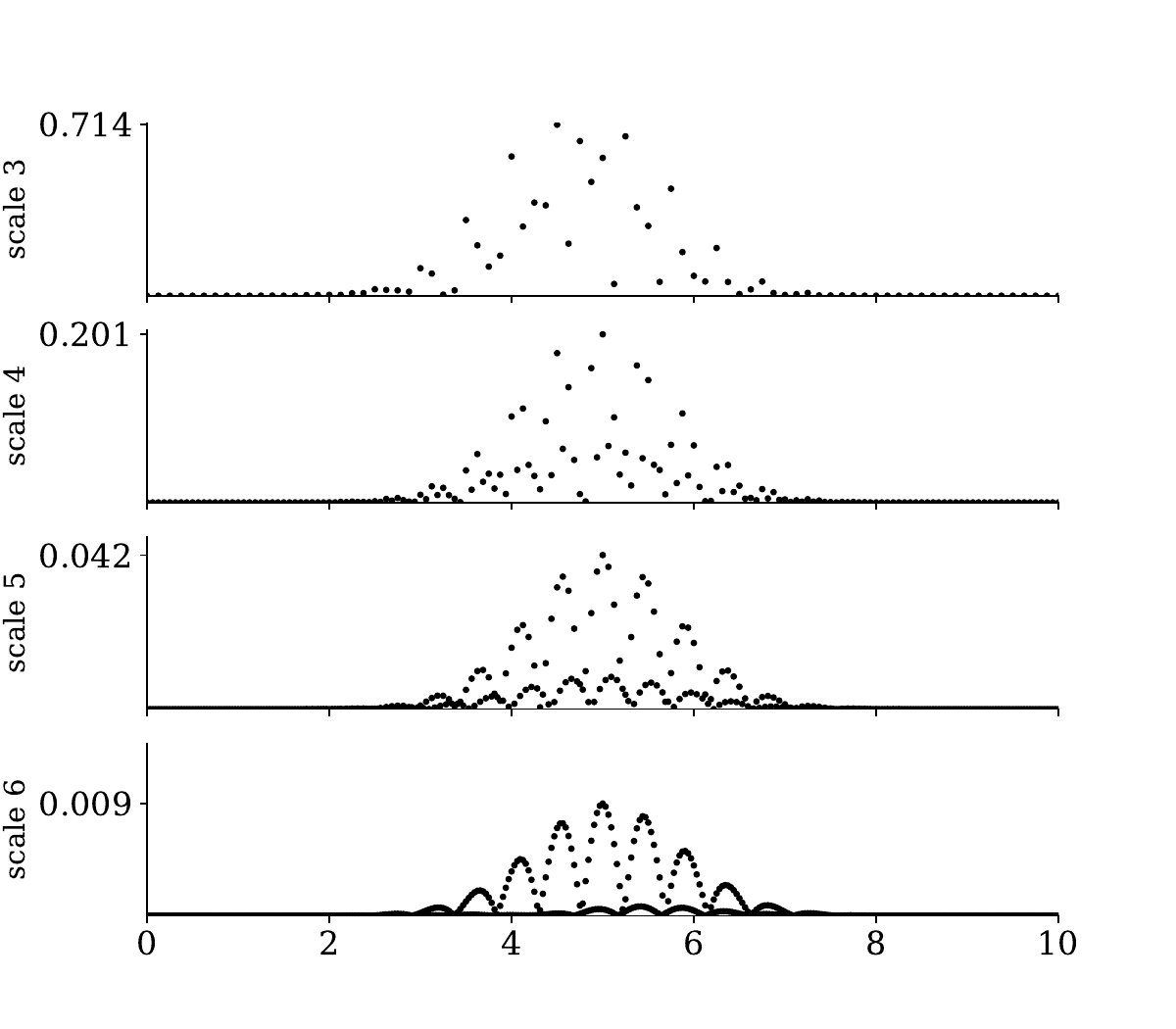}
        \caption{}
        \label{fig:linear_pyramid}
    \end{subfigure}
    \hfill
    \begin{subfigure}[b]{0.49\textwidth}   
        \includegraphics[width=\textwidth]{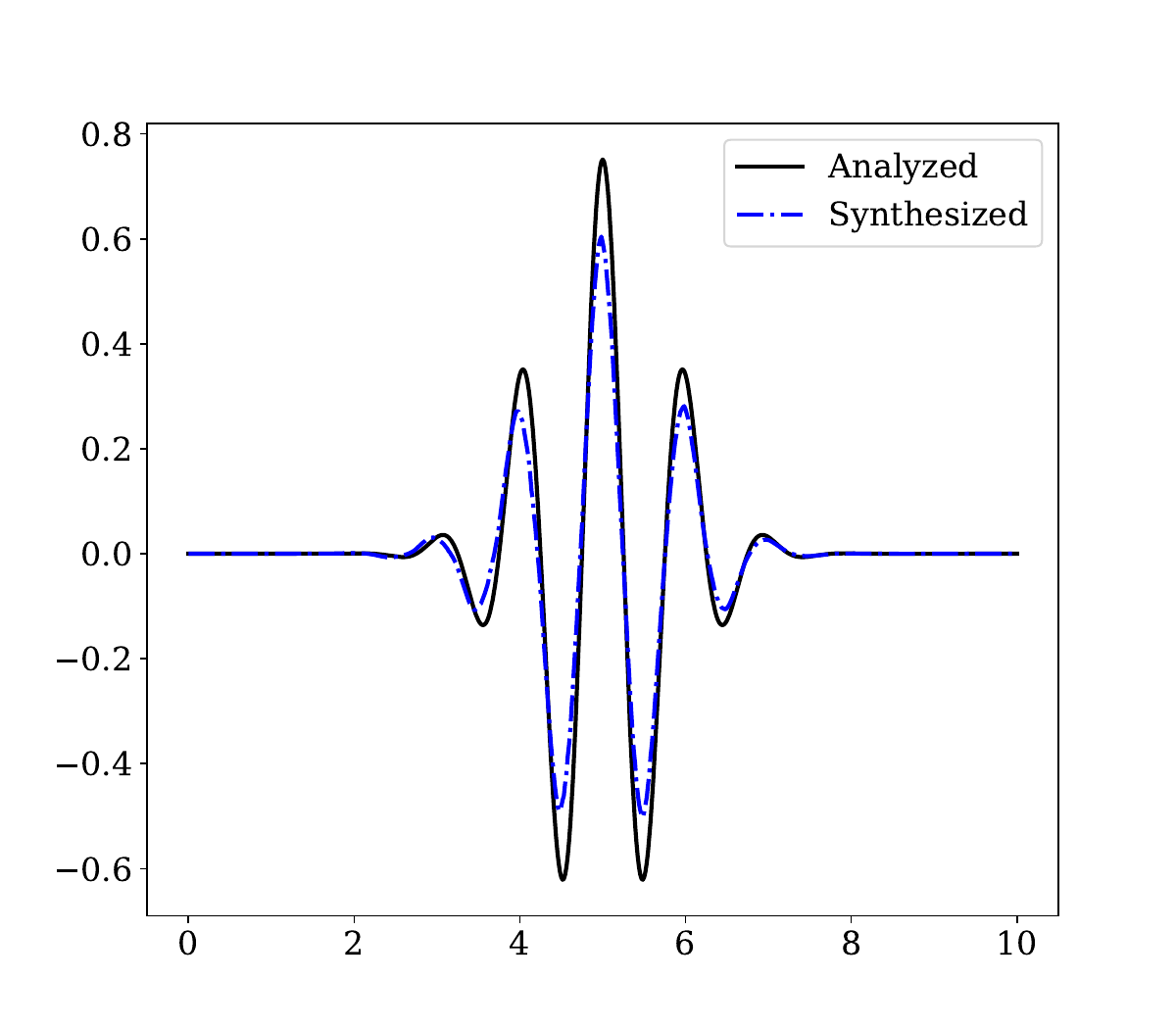}
        \caption{}
        \label{fig:linear_reconstruction}
    \end{subfigure}
    \vskip\baselineskip
    \begin{subfigure}[b]{0.49\textwidth}
        \includegraphics[width=\textwidth]{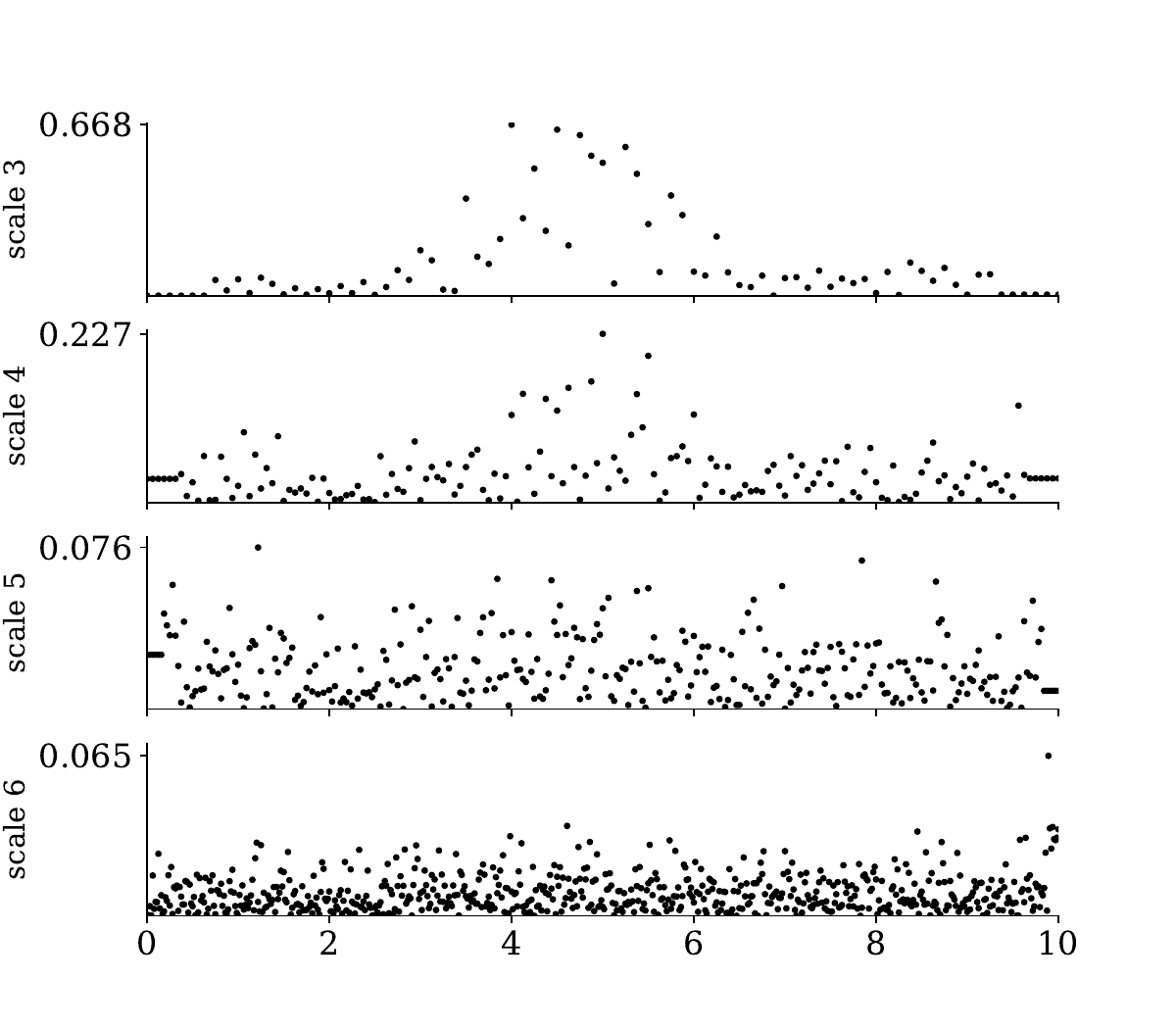}
        \caption{}
        \label{fig:noisy_linear_pyramid}
    \end{subfigure}
    \hfill
    \begin{subfigure}[b]{0.49\textwidth}   
        \includegraphics[width=\textwidth]{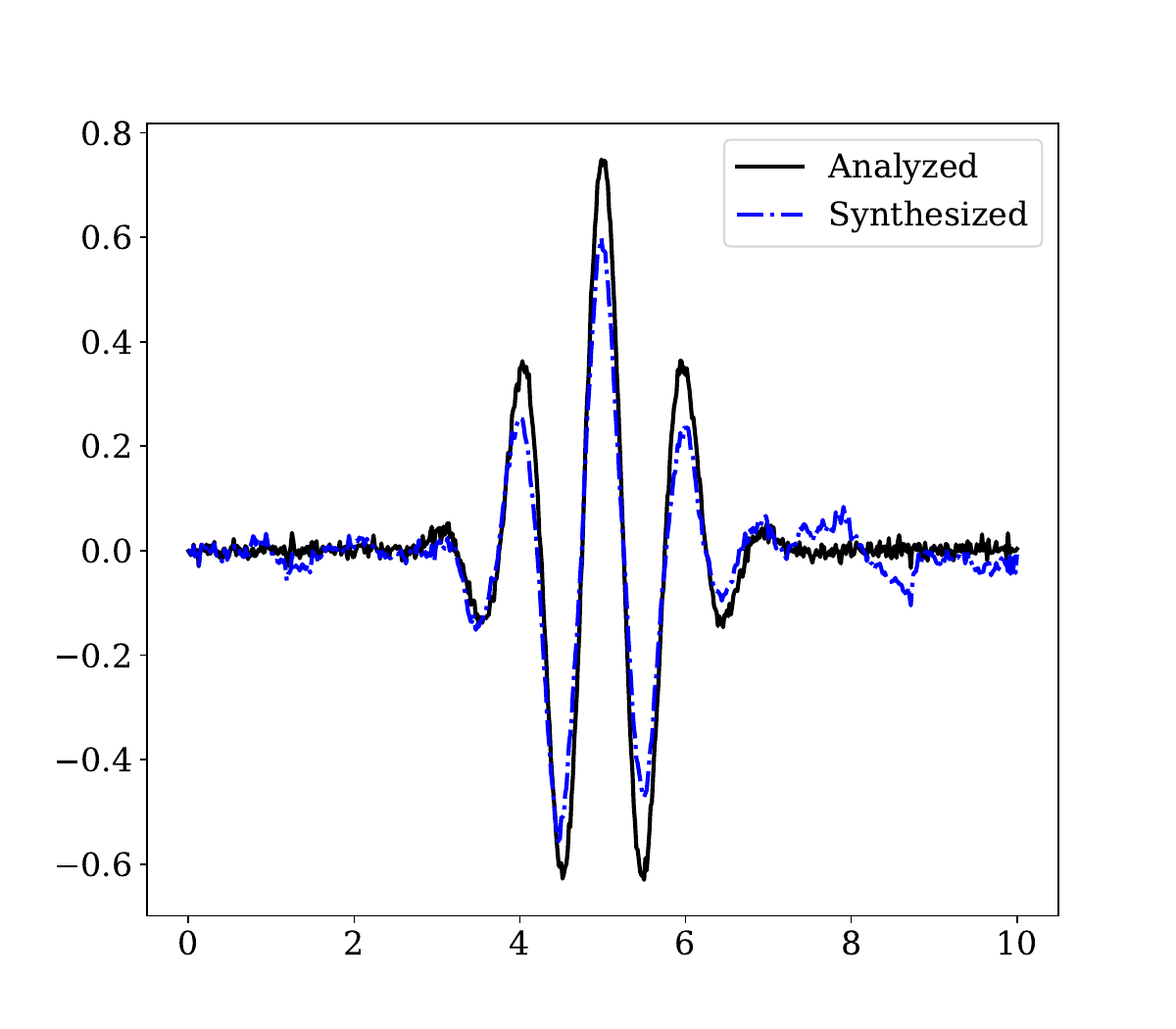}
        \caption{}
        \label{fig:noisy_linear_reconstruction}
    \end{subfigure}
    \caption{Linear multiscaling with a non-reversible subdivision scheme. On the right, the analyzed function in black, and the synthesized dot-dashed function in blue. On the left, the detail coefficients in absolute value of the first $4$ iterations of~\eqref{linear_multiscale_analysis} for $\mathcal{S}_{\boldsymbol{\alpha}}$ of~\eqref{LS_scheme_n_2_d_1} and its pseudo-reverse $\mathcal{D}_{\boldsymbol{\gamma}}$ for $\xi=1.4$.}
    \label{fig:linear_example}
\end{figure}

\begin{table}[H]
\begin{center}
    \begin{tabular}{| c || c | c | c | c | c | c |} 
     \hline
     $m$ & 1 & 2 & 3 & 4 & 5 & 6 \\
     \hline
     $\|\boldsymbol{c}^{(6)}-\boldsymbol{\zeta}^{(6)}\|_\infty$ for the smooth $\boldsymbol{c}^{(6)}$ & 0.0007 & 0.0059 & 0.0371 & 0.1896 & 0.3860 & 0.8816 \\ 
     \hline
     $\|\boldsymbol{c}^{(6)}-\boldsymbol{\zeta}^{(6)}\|_\infty$ for $\boldsymbol{c}^{(6)}$ with noise & 0.047 & 0.0492 & 0.0729 & 0.2045 & 0.4715 & 0.8756 \\ 
     \hline
    \end{tabular}
\end{center}
    \caption{The reconstruction error for the sequences of Figure~\ref{fig:linear_example} against the number of detail layers. Note that at $m=5,6$ the error becomes very high compared with the analyzed data. This drawback of our multiscaling~\eqref{linear_multiscale_analysis} is explained in Theorem~\ref{thm:synthesis_theorem} and its subsequent paragraphs.}
    \label{table:synthesis_error}
\end{table}

\subsection{Pyramid for SO(3) and manifold-valued contrast enhancement}

We begin with an illustration of the multiscale transform~\eqref{manifold_even_singular_multiscale} over the manifold of rotation matrices. That is, the rotation group $\text{SO}(3)$ acting on the Euclidean space $\mathbb{R}^3$. Then, we show the application of contrast enhancement using the multiscale representation.

Let $I_3$ be the identity matrix in $\mathbb{R}^{3\times 3}$, and denote by
\begin{equation}~\label{SO3_group}
    \text{SO}(3) = \left\{ R\in\mathbb{R}^{3\times 3} \;\big|\; R^TR= I_3, \; \text{det}(R) = 1\right\}
\end{equation}
the special orthogonal group consisting of all rotation matrices in $\mathbb{R}^{3}$. $\text{SO}(3)$ is endowed with a Riemannian manifold structure by considering it as a Riemannian submanifold of the embedding Euclidean space $\mathbb{R}^{3\times 3}$, with the inner product $\langle R_1, R_2\rangle = \text{trace}(R_1^TR_2)$ for $R_1,R_2\in\text{SO}(3)$.

The Riemannian geodesic connecting two matrices $R_1$ and $R_2$ in $\text{SO}(3)$ is parametrized by
\begin{equation}~\label{eqn:SO3_geodesic}
    \Gamma(s) = R_1(R_1^TR_2)^s, \quad s\in[0,1].
\end{equation}
Moreover, the Riemannian geodesic distance is explicitly computed by
\begin{equation}~\label{eqn:SO3_distance}
    \rho_{\text{SO(3)}}(R_1, R_2)=\|\text{Log}(R_1^TR_2)\|_{\text{F}},
\end{equation}
where $\|\cdot\|_\text{F}$ denotes the Frobenius norm, and $\text{Log}$ is the principal logarithm, see~\cite{moakher2002means} for more details. Further information can be found in~\cite{berger2012differential, curtis1979lie}.

One simple way to generate smooth and random $\text{SO}(3)$-valued sequences to test our multiscaling~\eqref{manifold_even_singular_multiscale} is to sample few rotation matrices, to associate the samples with indices, and then to refine using any refinement rule promising smooth limits, see e.g.,~\cite{wallner2011convergence}. Indeed, we followed this method to synthetically generate such a sequence. Specifically, we randomly generated $4$ rotation matrices, enriched the samples to $11$ matrices by a simple upsampling rule, and then refined the result using the cubic B-spline analogue~\eqref{manifold_subdivision} for a few iterations. The resulting sequence is then parametrized over the dyadic grid $2^{-6}\mathbb{Z}$ corresponding to scale 6. In the refinement process, the Riemannian center of masses were approximated by the method of geodesic inductive averaging presented in~\cite{dyn2017manifold}.

To visualize the generated $\text{SO}(3)$-valued sequence, we rotate the standard orthonormal basis of $\mathbb{R}^3$ using each rotation matrix, and then depict all results at different locations depending on the parametrization as a time series. Figure~\ref{fig:SO3_curve} illustrates the result.

We now analyze the synthetic $\text{SO}(3)$-valued curve appearing in Figure~\ref{fig:SO3_curve} via the multiscale transform~\eqref{manifold_even_singular_multiscale}, using the nonlinear analogue $\mathcal{T}_{\boldsymbol{\alpha}}$ of the subdivision scheme~\eqref{LS_scheme_n_2_d_1}, with its pseudo-reverse $\mathcal{Y}_{\boldsymbol{\gamma}}$ for $\xi=0.64$. The Riemannian center of masses~\eqref{manifold_subdivision} and~\eqref{manifold_decimation} were approximated by the method of geodesic inductive averaging. Figure~\ref{fig:SO3_detail_coefficients} exhibits the Frobenius norms of the first $4$ layers of detail coefficients. Note how the maximal norm of each layer can be bounded with geometrically decreasing values with respect to the scale. This indicates the smoothness of the curve. Moreover, the norms associated with the even indices are lower than the rest; this is a direct effect of $\boldsymbol{\gamma}$ truncation. Both phenomena were thoroughly explicated in~\cite{mattar2023pyramid}. Over and above, note how this effect is more pronounced on the high scales, which is explained by Lemma~\ref{lem:Manifold_details}.

We remark here that the detail coefficients generated by the multiscale transform~\eqref{manifold_even_singular_multiscale} lie in the tangent bundle $T\text{SO}(3)$ where the tangent space $T_R\text{SO}(3)$ of a rotation matrix $R$ is the set of all matrices $S\in\mathbb{R}^{3\times 3}$ such that $R^TS$ is skew-symmetric~\cite{edelman1998geometry}.

\begin{figure}[H]
\centering
\begin{subfigure}[b]{.49\textwidth}
  \centering
  \includegraphics[width=1\textwidth]{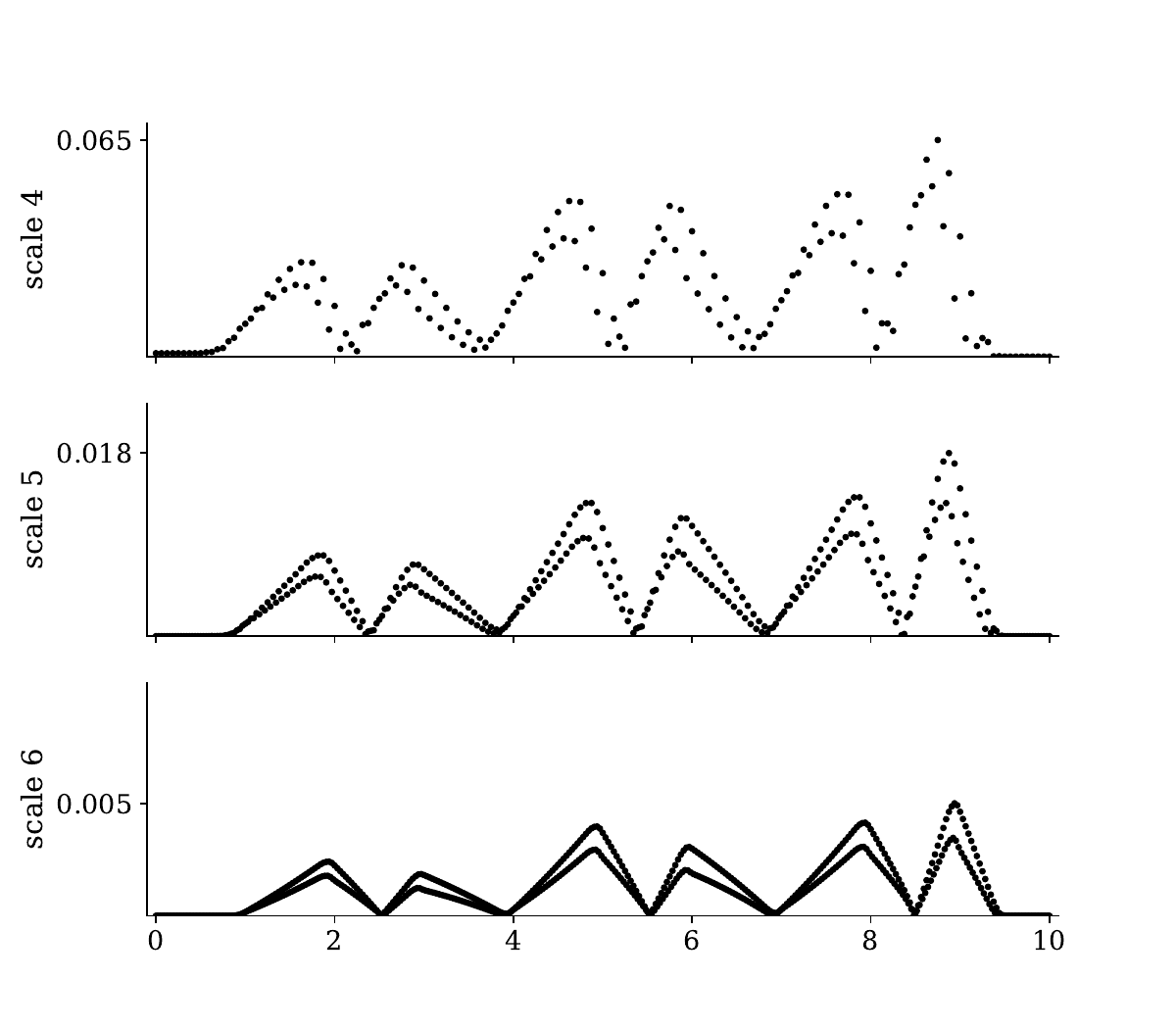}
  \caption{ }
  \label{fig:SO3_detail_coefficients}
\end{subfigure}%
\begin{subfigure}[b]{.49\textwidth}
  \centering
  \includegraphics[width=1\textwidth, trim=2.5cm 3.5cm 2.5cm 3cm, clip]{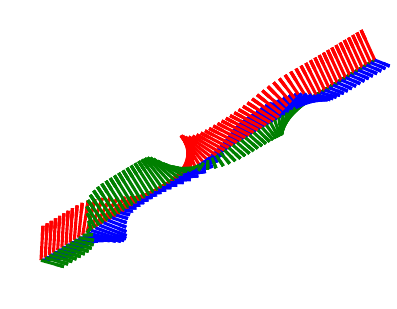}
  \caption{ }
  \label{fig:SO3_curve}
\end{subfigure}
\caption{Multiscaling of $\text{SO}(3)$-valued sequence with a non-reversible subdivision scheme. On the right, a visualization of a sequence of rotation matrices. On the left, the Frobenius norms of the last 3 layers of detail coefficients.}
\label{fig:SO3_pyramid}
\end{figure}

We now illustrate the application of contrast enhancement to the sequence in Figure~\ref{fig:SO3_curve} using its representation in Figure~\ref{fig:SO3_detail_coefficients}. In the context of multiscaling, the idea behind contrast enhancement lies in manipulating the detail coefficients while keeping the coarse approximation unchanged. Particularly, in order to add more contrast to a sequence, ``swerves'' in the case of rotations, one has to emphasize its most significant detail coefficients. Since all tangent spaces are closed under scalar multiplication, this can be done by multiplying the largest detail coefficients by a factor greater than $1$, while carefully monitoring the results to be within the injectivity radii of the manifold. The enhanced sequence is then synthesized from the modified pyramid.

Figure~\ref{fig:SO3_enhancement} shows the final result of enhancing the rotation sequence of Figure~\ref{fig:SO3_curve}, side by side, where the largest 20\% of the detail coefficients of each layer in Figure~\ref{fig:SO3_detail_coefficients} were scaled up by 40\%. Indeed, note how regions with small rotation changes are kept unchanged, while regions with high changes are more highlighted after the application. The specific percentages 20\% and 40\% were chosen to provide good and convincing visual results.

\begin{figure}[H]
\begin{subfigure}[b]{.49\textwidth}
  \centering
  \includegraphics[width=1\textwidth, trim=2.5cm 3.5cm 2.5cm 3cm, clip]{SO3_curve.pdf}
  \caption{ }
\end{subfigure}%
\begin{subfigure}[b]{.49\textwidth}
  \centering
  \includegraphics[width=1\textwidth, trim=2.5cm 3.5cm 2.5cm 3cm, clip]{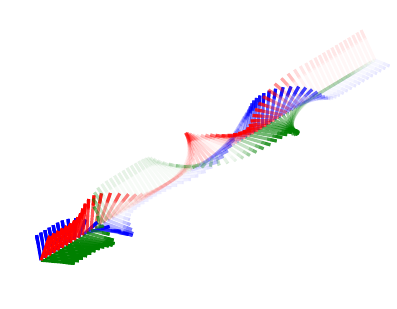}
  \caption{ }
\end{subfigure}
    \caption{Contrast enhancement of SO(3)-valued sequence. On the left, the original sequence of rotation matrices. On the right, the enhanced rotation sequence. The largest 20\% of the detail coefficients of each layer were enlarged by 40\%. To highlight the effect of the application, the color opacity of the arrows corresponds to the respective Riemannian distance between the original and the enhanced sequence. Regions where the application had less effect appear more transparent.}
    \label{fig:SO3_enhancement}
\end{figure}

\subsection{Data compression of rigid body motions}

Here we consider the special Euclidean Lie group $\text{SE}(3)$ of dimension 6, which is a semidirect product of the special orthogonal group $\text{SO}(3)$ in~\eqref{SO3_group} with the Euclidean space $\mathbb{R}^3$. Each element in this group makes a configuration of orientation and position to a rigid body. A convenient matrix representation of this group is
\begin{equation*}
\text{SE}(3)=\text{SO}(3)\ltimes\mathbb{R}^3 = \left\{\begin{pmatrix}
        R & p \\
        0 & 1
    \end{pmatrix}
\bigg| \; R\in\text{SO}(3), \; p\in\mathbb{R}^3
\right\}.
\end{equation*}
The space $\text{SE}(3)$ exhibits a Riemannian manifold structure when equipped with a left-invariant Riemannian metric~\cite{doCarmo1992, kirillov2008introduction}. In particular, we choose the left-invariant Riemannian metric defined via $\langle V_1,V_2\rangle_A=\text{trace}((A^{-1}V_1)^T(A^{-1}V_2))$ for any $A\in\text{SE}(3)$ and $V_1,V_2\in T_A\text{SE}(3)$. Moreover, the tangent space at the identity of the group, that is the identity matrix $I_4$, is given by the following set
\begin{equation*}
    T_{I_4}\text{SE}(3) = \bigg\{\begin{pmatrix}
        \Omega & v \\
        0 & 0
    \end{pmatrix}\;\bigg|\;\Omega\in\mathbb{R}^{3\times 3},\;\Omega^T = -\Omega,\;v\in\mathbb{R}^3\bigg\}.
\end{equation*}

The Riemannian geodesic connecting two elements $A_1,A_2\in\text{SE}(3)$ that are comprised of two rotation matrices $R_1,R_2\in\text{SO}(3)$ and points $p_1, p_2\in\mathbb{R}^3$, respectively, is given by
\begin{equation}
    \Gamma(s)=\begin{pmatrix}
        R_1(R_1^TR_2)^s & p_1+s(p_2-p_1) \\
        0 & 1
    \end{pmatrix}, \quad s\in[0,1].
\end{equation}
Furthermore, the Riemannian geodesic distance is provided by the formula
\begin{equation}~\label{eqn:SE3_distance}
    \rho_{\text{SE(3)}}(A_1, A_2) = \sqrt{\rho_{\text{SO(3)}}^2(R_1, R_2)+\|p_2-p_1\|^2_2},
\end{equation}
where $\rho_{\text{SO(3)}}$ is the Riemannian distance over the manifold $\text{SO}(3)$ as appears in~\eqref{eqn:SO3_distance}. For an overview of these results, see~\cite{duan2013riemannian, vzefran1996choice}.

In this subsection, we take the $\text{SO}(3)$-valued curve appearing in Figure~\ref{fig:SO3_curve} and wrap it around a cone in $\mathbb{R}^3$, obtaining a time-series of location coordinates and orientations, that is, a curve on $\text{SE}(3)$ parametrized over the equispaced grid $2^{-6}\mathbb{Z}$. Figure~\ref{fig:SE3_curve} depicts the curve and is often called a rigid body motion. We aim to show that there is a slight visual difference, as well as small empirical errors, between the original curve and the synthesized one.

The value of this experiment is to show that it is possible to use our pyramid representation for the task of data compression. Similar to compression techniques using wavelets~\cite{ning2010wavelet}, in this application, one usually sets to zero a prefixed percentage of the smallest detail coefficients. Or equivalently, all coefficients with norms below a certain threshold. A sparser representation is then obtained, which can be synthesized into the compressed result. The validity of the experiment lies in the ability to omit a significant percentage of the information in the multiscale representation while maintaining visual resemblance to the original curve after synthesis.

Let us now present the technical aspects of the experiment. We analyze the $\text{SE}(3)$-valued curve with the multiscale transform~\eqref{manifold_even_singular_multiscale} using the nonlinear analogue $\mathcal{T}_{\boldsymbol{\alpha}}$~\eqref{manifold_subdivision} of the subdivision scheme~\eqref{LS_scheme_n_2_d_1} with its pseudo-reverse $\mathcal{Y}_{\boldsymbol{\gamma}}$ of~\eqref{manifold_decimation} for $\xi=0.64$. Specifically, we decompose the sequence four times to obtain the pyramid $\{ \boldsymbol{c}^{(2)}; \boldsymbol{d}^{(3)}, \boldsymbol{d}^{(4)}, \boldsymbol{d}^{(5)},\boldsymbol{d}^{(6)}\}$. The Riemannian center of masses~\eqref{manifold_subdivision} and~\eqref{manifold_decimation} were approximated by the method of geodesic inductive averaging~\cite{dyn2017manifold} using $\rho_{\text{SE}(3)}$ of~\eqref{eqn:SE3_distance}. For compression, we set to zero $99\%$ of the detail coefficients, which is a staggering percentage! In practice, instead of storing $641$ ground-truth SE(3)-valued matrices, we can compress and store $41$ SE(3)-valued matrices in addition to only $12$ detail coefficients!

Figure~\ref{fig:SE3_compression} depicts the original sequence in conjunction with the synthesized and compressed result. Note that if one chooses to store more detail coefficients or decompose into fewer layers, then the compression result will yield smaller errors, as explained by Theorem~\ref{ManifoldStabilityTHM}. In the depicted example, we measure the error between the original sequence $\boldsymbol{c}^{(6)}$ and its estimation $\boldsymbol{\zeta}^{(6)}$ using an element-wise geodesic error, defined as 
\begin{equation}~\label{eqn:geodesic_error}
    \rho_{\text{SE(3)}}(c^{(6)}_j,\zeta^{(6)}_j).
\end{equation}
Figure~\ref{fig:Riemannian_distance_histogram} demonstrates the distribution of these errors with a brief explanation of the result.

\begin{figure}[H]
\centering
\begin{subfigure}[c]{0.49\textwidth}
  \includegraphics[width=1\textwidth]{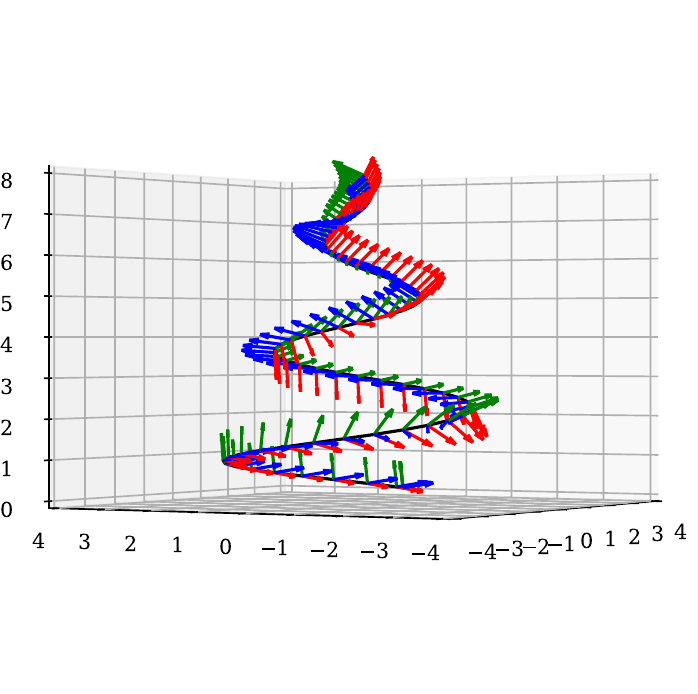}
  \caption{Original SE(3) curve.}
  \label{fig:SE3_curve}
\end{subfigure}\hfil
\begin{subfigure}[c]{0.49\textwidth}
  \includegraphics[width=1\textwidth]{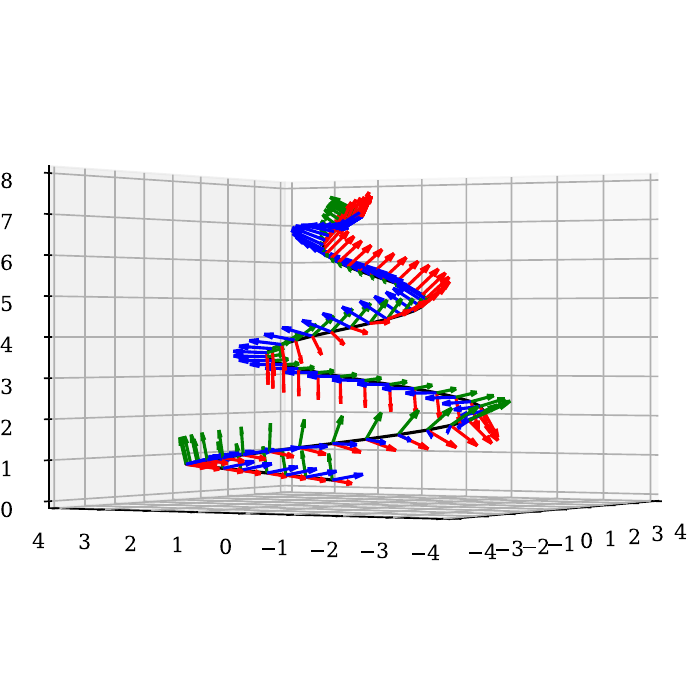}
  \caption{Compressed curve.}
\end{subfigure}
\caption{Compression of SE(3)-valued curve. On the left is a depiction of the original ground-truth curve. On the right is the result of synthesis using the coarsest approximation of the curve $\boldsymbol{c}^{(0)}$ in addition to just $1\%$ of the detail coefficients. Quantitatively, the original curve stores $641$ SE(3) matrices while the compression holds only $53$ matrices, $41$ of which are SE(3)-valued. The difference between the two curves is visually noticeable mostly in the top and the bottom parts.}
\label{fig:SE3_compression}
\end{figure}

\begin{figure}[H]
    \centering
    \includegraphics[width=0.5\linewidth]{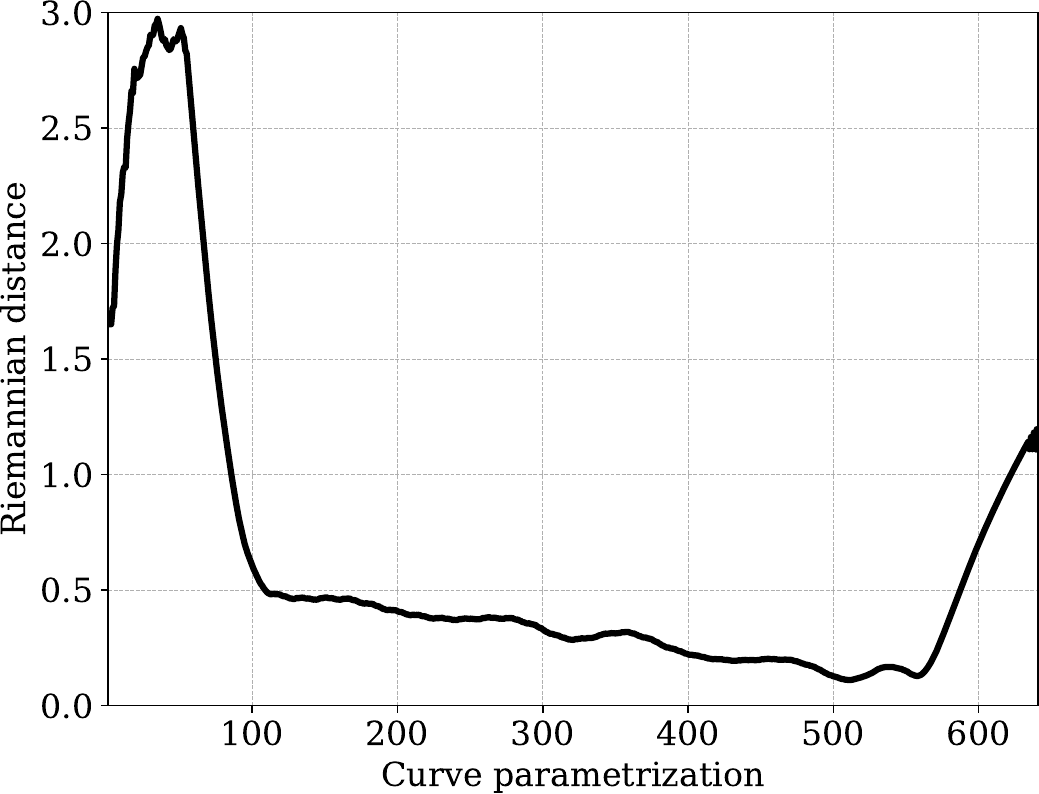}
    \caption{The pairwise Riemannian distances~\eqref{eqn:geodesic_error} between the original and the compressed curve. Because the compression error is visually apparent at the top and bottom parts of the curve, see Figure~\ref{fig:SE3_compression}, the error plot is high around the endpoints of the parametrization.}
    \label{fig:Riemannian_distance_histogram}
\end{figure}

%======================================

% \section*{Conclusion}
% In this paper, we introduced a new method to approximate non-invertible elements in the Banach algebra of functions with absolutely summable Fourier coefficients. The method is called pseudo-reversing and is based on the Wiener's lemma. In particular, our approximation is achieved simply by displacing the ``problematic'' zeros of the given function. Although we studied this idea from the perspective of elementary calculus, its implications in advanced analyses remain unknown, and perhaps deserve the effort of a follow-up research, considering the significance of celebrated Wiener's lemma in harmonic analysis.

% Using pseudo-reversing, we then constructed a novel multiscale transform as a tool to analyze univariate sequences of data points, both in the real- and manifold-valued settings. The benefit of pseudo-reversing comes into play when the transform is based on non-reversible upsampling refinements. On the other hand, a drawback of using pseudo-reversing is manifested when analyzing sequences with rapid consecutive changes, such as noisy data. In this case, the synthesis is susceptible to large errors. Additional research is required to overcome this disadvantage.

\section*{Acknowledgments}
The authors would like to extend their sincere gratitude to the peer reviewers for their valuable time, insightful feedback and constructive suggestions. Their thoughtful comments and recommendations have greatly contributed to improving the quality of this work.

N. Sharon is partially supported by the NSF-BSF award 2019752 and the DFG award 514588180. W. Mattar is partially supported by the Nehemia Levtzion Scholarship for Outstanding Doctoral Students from the Periphery (2023).

%======================================

\bibliographystyle{plain}
\bibliography{references}
\end{document}